\documentclass[journal,onecolumn,11pt]{IEEEtran}
\usepackage{amsfonts,color,morefloats}
\usepackage{amssymb,amsmath,latexsym,amsthm}
\usepackage{hyperref}
\hypersetup{hidelinks}

\usepackage{booktabs}
\usepackage{amsmath,amsfonts,amsthm,amssymb}
\usepackage{algorithmic}
\usepackage{algorithm}
\usepackage{array}
\usepackage[caption=false,font=normalsize,labelfont=sf,textfont=sf]{subfig}
\usepackage{stfloats}
\usepackage{url}
\usepackage{verbatim}
\usepackage{graphicx}
\usepackage[noadjust]{cite}
\usepackage{comment}
\usepackage{enumitem}
\usepackage{listings}
\usepackage{xcolor} 

\newcommand{\mymod}{{\rm{mod}}\,}

\allowdisplaybreaks

\theoremstyle{remark}

\newtheorem{definition}{\indent Definition}
\newtheorem{lemma}{\indent Lemma}
\newtheorem{corollary}{\indent Corollary}

\newtheorem{theorem}{\indent Theorem}
\newtheorem*{remark}{\indent Remark}

\begin{document}

\title{The Differential  and Boomerang Properties of a Class of Binomials}

\author{Sihem~Mesnager~\textit{Member, IEEE}~and~Huawei~Wu
  \thanks{S. Mesnager is with the Department of Mathematics, University of Paris VIII, 93526 Saint-Denis, France, also with the Laboratory Analysis, Geometry and Applications (LAGA), UMR 7539, CNRS, University Sorbonne Paris Cit\'e, 93430 Villetaneuse, France, and also with Telecom Paris, Polytechnic Institute of Paris, 91120 Palaiseau, France (e-mail: smesnager@univ-paris8.fr)}%
  \thanks{H. Wu is with the Department of Mathematics, University of Paris VIII, 93526 Saint-Denis, France, and also with the Laboratory Analysis, Geometry and Applications (LAGA), UMR 7539, CNRS, University Sorbonne Paris Cit\'e, 93430 Villetaneuse, France (e-mail: wuhuawei1996@gmail.com)}
}

\markboth{}%
{Shell \MakeLowercase{\textit{et al.}}: A Sample Article Using IEEEtran.cls for IEEE Journals}


\maketitle

\begin{abstract}
  Let $q$ be an odd prime power with $q\equiv 3\ (\mymod 4)$. In this paper, we study the differential and boomerang properties of the function $F_{2,u}(x)=x^2\big(1+u\eta(x)\big)$ over $\mathbb{F}_{q}$, where $u\in\mathbb{F}_{q}^*$ and $\eta$ is the quadratic character of $\mathbb{F}_{q}$. We determine the differential uniformity of $F_{2,u}$ for any $u\in\mathbb{F}_{q}^*$ and determine the differential spectra and boomerang uniformity of the locally-APN functions $F_{2,\pm 1}$, thereby disproving a conjecture proposed in \cite{budaghyan2024arithmetization} which states that there exist infinitely many $q$ and $u$ such that $F_{2,u}$ is an APN function.
\end{abstract}

\begin{IEEEkeywords}
  Boomerang uniformity, boomerang spectrum, character sums, differential uniformity, differential spectrum, locally-APN functions, Ness-Helleseth functions, Weil bound
\end{IEEEkeywords}

\section{Introduction}\label{20240907section0}
\IEEEPARstart{L}{et} $\mathbb{F}_{q}$ be the finite field with $q$ elements, where $q=p^n$, $p$ is a prime and $n$ is a positive integer. For any function $f$ over $\mathbb{F}_q$ and any element $a\in\mathbb{F}_q$, the derivative of $f$ at $a$ is defined as
$$D_af(x)=f(x+a)-f(x),\quad x\in\mathbb{F}_q.$$
For any $a,b\in\mathbb{F}_q$, let
$$\delta_f(a,b)=\#\{x\in\mathbb{F}_q:\ D_af(x)=b\}.$$
The differential uniformity of $f$ is defined as
$$\delta_f=\max\limits_{\substack{a\in\mathbb{F}_q^*\\b\in\mathbb{F}_q}}\delta(a,b),$$
which was introduced by Nyberg in \cite{nyberg1993differentially} to measure the ability of $f$, when used as an S-box (substitution box) in a cipher, to resist differential attacks. The differential uniformity is desired to be as low as possible, corresponding to a stronger resistance against differential attacks. If $\delta_f=1$, then $f$ is called a perfect nonlinear (PN) function, which exists only in odd characteristics. Whereas, if $\delta_f=2$, then $f$ is called an almost perfect nonlinear (APN) function, which is the minimum possible value for binary fields.
When studying the differential properties of a function $f$, the differential uniformity is the most basic characteristic that needs to be determined. The differential spectrum of $f$ can provide more detailed information on the differential properties of $f$, which is defined as the following multiset
$${\rm{DS}}_f=\{\omega_i:\ 0\le i\le\delta_f\},$$
where
$$\omega_i=\#\{(a,b)\in\mathbb{F}_q^*\times\mathbb{F}_q:\ \delta_f(a,b)=i\}.$$
We have the following fundamental property of the differential spectrum (see, for instance, \cite{blondeau2010differential}):
\begin{equation}\label{20240622equation1}
  \sum\limits_{i=0}^{\delta_f}\omega_i=\sum\limits_{i=0}^{\delta_f}i\omega_i=(q-1)q.
\end{equation}

In \cite{blondeau2011differential}, when working on the differential properties of power functions over binary fields, Blondeau and Nyberg introduced a new concept called locally-APNness. They showed that a locally-APN S-box could achieve lower differential probabilities compared to S-boxes with differential uniformity $4$, using a cryptographic toy example \cite{blondeau2015perfect}. Recently, Hu et al. generalized this concept to general functions over arbitrary finite fields \cite{hu2023differential}. A function $f$ over $\mathbb{F}_{q}$ is said to be locally-APN if
$$\max\{\delta_f(a,b):\ a\in\mathbb{F}_q^*,\ b\in\mathbb{F}_{q}\setminus\mathbb{F}_p\}=2.$$

The boomerang attack is a variant of the differential attack proposed by Wagner in \cite{wagner1999boomerang}, which combines the differential layers of the upper and lower layers of block ciphers. The quantity measures the resistance of an S-box to boomerang attacks is called the boomerang uniformity, which was introduced by Boura and Canteaut in \cite{boura2018boomerang} for permutations over binary fields and was later generalized to general functions over arbitrary finite fields by Li et al. in \cite{li2019new}. The boomerang uniformity of a function $f$ over $\mathbb{F}_q$ is defined as
$$\beta_f=\max\limits_{a,b\in\mathbb{F}_q^*}\beta_f(a,b),$$
where $\beta_f(a,b)$ denotes the number of solutions $(x,y)\in\mathbb{F}_q^2$ to the following system of equations
$$\begin{cases}
    f(x)-f(y)=b, \\
    f(x+a)-f(y+a)=b.
  \end{cases}$$
Similarly, the boomerang spectrum of $f$ is defined as the following multiset
$${\rm{BS}}_f=\{\nu_i:\ 0\le i\le\beta_f\},$$
where
$$\nu_i=\#\{(a,b)\in\mathbb{F}_q^*\times\mathbb{F}_q^*:\ \beta_f(a,b)=i\}.$$

From here until the end of this section, we assume that $q\equiv 3\ (\mymod 4)$. Let $C_0$ (resp., $C_1$) denote the set of non-zero square (resp., non-square) elements in $\mathbb{F}_q$. It is known that there exists a unique quadratic character $\eta$ of $\mathbb{F}_q$ which is given by
\begin{align*}
  \eta(x)=\left\{\begin{array}{ll}
                   0,  & \text{if }x=0,      \\
                   1,  & \text{if }x\in C_0, \\
                   -1, & \text{if }x\in C_1.
                 \end{array}\right.
\end{align*}
Throughout this paper, the symbol $\eta$ always represents this meaning.

Consider the following function over $\mathbb{F}_q$:
$$F_u(x)=ux^{\frac{q-3}{2}}+x^{q-2},$$
where $u\in\mathbb{F}_q$. It was first studied by Ness and Helleseth in \cite{ness2007new} for the ternary case and was later generalized to the general case by Zeng et al. in \cite{zeng2007inequivalence}. They showed that if $\eta(u+1)=\eta(u-1)=-\eta(5u+3)$ or $\eta(u+1)=\eta(u-1)=-\eta(5u-3)$, then $F_u$ is an APN function. Subsequently, several papers have been dedicated to studying the differential properties of $F_u$. Zha proved in his PhD dissertation \cite{zha2008} that the differential uniformity of $F_u$ is $3$ if $\eta(u+1)=\eta(u-1)=\eta(5u+3)=\eta(5u-3)$. Recently, when $p=3$, Xia et al. \cite{xia2024more} determined the differential uniformity of $F_u$ for any $u\in\mathbb{F}_q$ and expressed the differential spectrum of $F_u$ in terms of several quadratic character sums of cubic polynomials for any $u\in\mathbb{F}_q$ with $\eta(u+1)=\eta(u-1)$. Very recently, they generalized in \cite{xia2024further} their results to the case of a general odd power $q$ satisfying $q\equiv 3\ (\mymod{4})$. It is worth mentioning that when $u=1$ or $-1$, although $F_{u}$ has a large differential uniformity (equaling $\frac{q+1}{4}$), it is locally-APN.
This was first observed by Lyu et al. in \cite{lyu2024further}, where they also computed the boomerang spectra of $F_{\pm 1}$, revealing the first class of non-PN functions whose boomerang uniformity can attain $0$ or $1$.

Note that $F_u$ can be rewritten as $F_u(x)=x^{p^n-2}\big(1+u\eta(x)\big)$. This inspires us to consider the following generalization of $F_u$:
\begin{equation}\label{20240907equation19}
  F_{r,u}(x)=x^r\big(1+u\eta(x)\big),
\end{equation}
where $r\in\mathbb{N}_+$ and $u\in\mathbb{F}_q$. The numerical results indicate that many of the $F_{r,u}$'s exhibit low differential uniformity. In this paper, we study what appears to be the simplest case, $r=2$. Since $F_{2,0}$ is the square function, which has been extensively studied, we always assume that $u\ne 0$. In \cite{budaghyan2024arithmetization}, Budaghyan and Pal showed that $\delta_{F_{2,u}}\le 5$ for any $u\in\mathbb{F}_q\setminus\{0,\pm1\}$. Moreover, based on their computational results over fields of small orders, they conjectured that there exist infinitely many $q$ and $u$ such that $F_{2,u}$ is an APN function. In this paper, we show that the conjecture does not hold by determining the differential uniformity of the functions $F_{2,u}$'s. 

We set the following set
$$\mathcal{U}=\begin{cases}
    \{0,\pm 1\}                & \mbox{if}\ p=3,    \\
    \{0,\pm 1,\pm\frac{1}{3}\} & \mbox{if}\ p\ne 3.
  \end{cases}$$
The remainder of the paper is organized as follows. In Section \ref{20240907section1}, we introduce some basic concepts and several results that will be used later. In Section \ref{20240907section2}, we preliminarily investigate the differential properties of $F_{2,u}$ for $u\in\mathbb{F}_q\setminus\{0,\pm 1\}$. In Section \ref{20240907section3}, we determine the differential uniformity of $F_{2,u}$ for $u\in\mathbb{F}_q\setminus\mathcal{U}$. In Section \ref{20240907section4}, we determine the differential uniformity of $F_{2,\pm\frac{1}{3}}$ when $p\ne 3$. In Section \ref{20240907section5}, we determine the differential spectra and boomerang uniformity of $F_{2,\pm 1}$. Finally, Section \ref{20240907section6} serves as a conclusion.

\section{Preliminaries}\label{20240907section1}

If $q$ is an odd prime power, we use $C_0$ (resp., $C_1$) to denote the set of non-zero square (resp., non-square) elements in $\mathbb{F}_q$. If $s\in C_0$, we use $\pm\sqrt{s}$ to denote the two square roots of $s$ in $\mathbb{F}_q$. Put
$$\begin{cases}
    C_{00}=\{x\in\mathbb{F}_{q}:\ \eta(x)=\eta(x+1)=1\},      \\
    C_{01}=\{x\in\mathbb{F}_{q}:\ \eta(x)=1,\ \eta(x+1)=-1\}, \\
    C_{10}=\{x\in\mathbb{F}_{q}:\ \eta(x)=-1,\ \eta(x+1)=1\}, \\
    C_{11}=\{x\in\mathbb{F}_{q}:\ \eta(x)=\eta(x+1)=-1\}.
  \end{cases}$$
Then $\mathbb{F}_q=C_{00}\cup C_{01}\cup C_{10}\cup C_{11}\cup\{0,-1\}$.

Regarding the sizes of the sets $C_{ij}$'s, we have the following conclusion.

\begin{lemma}[{\cite[Lemma 6]{storer1967}}]\label{20240908lemma1}
  For any $i,j\in\{0,1\}$, put $(i,j)=\# C_{ij}$.
  \begin{enumerate}
    \item If $q\equiv 1\ (\mymod 4)$, then
          $$(0,0)=\frac{q-5}{4},\ (0,1)=(1,0)=(1,1)=\frac{q-1}{4};$$
    \item If $q\equiv 3\ (\mymod 4)$, then
          $$(0,0)=(1,0)=(1,1)=\frac{q-3}{4},\ (0,1)=\frac{q+1}{4}.$$
  \end{enumerate}
\end{lemma}

The following helpful lemma will be used repeatedly later.

\begin{lemma}\label{20240824lemma3}
  Assume that $q\equiv 3\ (\mymod{4})$. Let $a\in\mathbb{F}_q$ and $u,u'\in C_0$ be such that $u+u'=a^2$. Suppose that $\epsilon\in\{\pm1\}$. Then $\eta(a\pm\sqrt{u})=\epsilon$ if and only if $\eta(a\pm\sqrt{u'})=\eta(2)\epsilon$.
\end{lemma}
\begin{proof}
  By symmetry, it suffices to prove the sufficiency. We assume that $\eta(a\pm\sqrt{u'})=\eta(2)\epsilon$. Since $(a+\sqrt{u})(a-\sqrt{u})=a^2-u=u'\in C_0$, we have $\eta(a\pm\sqrt{u})=\epsilon$ or $\eta(a\pm\sqrt{u})=-\epsilon$. For a contradiction, assume that $\eta(a+\sqrt{u})=-\epsilon$. Then there exists $t\in\mathbb{F}_q$ such that $-\epsilon t^2=a+\sqrt{u}$, which implies that
  \begin{equation}\label{20240824equation2}
    (-\epsilon t^2-a)^2=u\iff t^4+2\epsilon at^2+a^2-u=0.
  \end{equation}
  Consider the quadratic polynomial $x^2+2\epsilon ax+a^2-u$, whose discriminant is $4u\in C_0$. Hence, it has two distinct roots $x_1$ and $x_2$ in $\mathbb{F}_q$. Since $x_1x_2=a^2-u\in C_0$, both $x_1$ and $x_2$ are square or neither is square. It follows that equation (\ref{20240824equation2}) has four distinct solutions or no solution in $\mathbb{F}_q$. By assumption, it has at least one solution in $\mathbb{F}_q$, so it has four distinct solutions in $\mathbb{F}_q$, i.e., $\pm\sqrt{x_1}$ and $\pm\sqrt{x_2}$. Note that
  \begin{align}\label{20240907equation1}
        & t^4+2\epsilon at^2+a^2-u                                 \notag \\
    =\  & (t+\sqrt{x_1})(t-\sqrt{x_1})(t+\sqrt{x_2})(t-\sqrt{x_2}) \notag \\
    =\  & (t^2+ct+d)(t^2-ct+d),
  \end{align}
  where $c=\sqrt{x_1}+\sqrt{x_2}$ and $d=\sqrt{x_1}\sqrt{x_2}$. Comparing the coefficients of $t^2$ and $t^0$, we have $-c^2+2d=2\epsilon a$ and $d^2=a^2-u=u'$, which implies that $d=\sqrt{u'}$ or $d=-\sqrt{u'}$. Moreover, since polynomial (\ref{20240907equation1}) has four distinct roots in $\mathbb{F}_q$, we have $\eta(c^2-4d)=1$. However, we have $c^2-4d=2(d-\epsilon a)-4d=-2(d+\epsilon a)=-2\epsilon(a+\epsilon d)$, which implies that $\eta(a+\epsilon d)=-\eta(2)\epsilon$ (note that $\eta(-1)=-1$ since $q\equiv 3\ (\mymod 4)$). This contradicts our assumption that $\eta(a\pm\sqrt{u'})=\eta(2)\epsilon$. Hence $\eta(a\pm\sqrt{u})=\epsilon$.
\end{proof}

An important tool used later in this paper is estimating rational points on algebraic curves over finite fields. Here, we only provide a minimal introduction to the necessary concepts.

Let $q$ be an arbitrary prime power. For any polynomial $f\in\mathbb{F}_q[x_1,\cdots,x_n]$, we use $A_f(\mathbb{F}_q)$ to denote the zero set of $f$ in $\mathbb{F}_q^n$, i.e.,
$$A_f(\mathbb{F}_q)=\{(a_1,\cdots,a_n)\in\mathbb{F}_q^n:\ f(a_1,\cdots,a_n)=0\}.$$

We need the following definition of absolute irreducibility.
\begin{definition}
  Let $F$ be a field. A polynomial $f\in F[x_1,\cdots,x_n]$ is said to be absolutely irreducible if it is irreducible over $\overline{F}$, the algebraic closure of $F$.
\end{definition}

The following is a variation of the well-known Weil estimate.
\begin{theorem}[{\cite[Theorem 7.1.9]{mullen2013handbook}}]\label{20240904them1}
  If $f\in\mathbb{F}_q[x_1,\cdots,x_n]$ is an absolutely irreducible polynomial of degree $d>0$, then
  $$\left|\# A_f(\mathbb{F}_q)-q^{n-1}\right|\le(d-1)(d-2)q^{n-\frac{3}{2}}+5d^{\frac{13}{3}}q^{n-2}.$$
\end{theorem}

Let $F$ be a field. Note that a bivariate polynomial $f(x,y)\in F[x,y]$ can be viewed as a univariate polynomial over the ring $F[x]$. If $f(x,y)$ has the form $f(x,y)=y^k+f_{k-1}(x)y^{k-1}+\cdots+f_0(x)$, where $f_i(x)\in F[x]$ for $0\le i\le k-1$, then the irreducibility of $f$ in $F[x,y]$ is equivalent to its irreducibility as a polynomial over $F[x]$. Thus, in this case, we only need to consider the irreducibility of univariate polynomials over rings. The following two simple lemmas will play a crucial role in verifying the absolute irreducibility of bivariate polynomials later.

\begin{lemma}\label{20240904lemma2}
  Let $R$ be an integral domain and $A, B\in R$. If neither $A^2-4B$ nor $B$ is a square element in $R$, then the polynomial $f(x)=x^4+Ax^2+B$ is irreducible in $R[x]$.
\end{lemma}

\begin{proof}
  Assume that $f(x)$ can be factored as the product of two quadratic polynomials, i.e.,
  $$f(x)=(x^2+A_1x+B_1)(x^2+A_2x+B_2).$$
  Comparing the coefficients of $x^3$ and $x$ on both sides, we have
  $$A_1+A_2=0,\qquad A_1B_2+A_2B_1=0.$$
  If $A_1=A_2=0$, then we have $B_1+B_2=A$ and $B_1B_2=B$, which implies that $A^2-4B=(B_1-B_2)^2$ is a square element in $R$. This contradicts our hypothesis. Hence, both $A_1$ and $A_2$ are non-zero, which implies that $B_1=B_2$. Then, $B=B_1^2$ is a square element in $R$, which contradicts our hypothesis. Hence, $f(x)$ cannot be factored as the product of two quadratic polynomials.

  Now assume that $f(x)$ can be factored as the product of a linear and cubic polynomial. Then $f(x)$ has a root in $R$, say $a$, which implies that $x_1=a^2$ is a root of the quadratic polynomial $\tilde{f}(x)=x^2+Ax+B$. It follows that $\tilde{f}(x)$ has another root $x_2$ in $R$, with $x_1+x_2=-A$ and $x_1x_2=B$. Then, we again have $A^2-4B$, a square element in $R$, which contradicts our hypothesis. Hence $f(x)$ is irreducible in $R[x]$.
\end{proof}

\begin{lemma}\label{20240904lemma1}
  Let $R$ be an integral domain such as $2\ne 0$ and $A, B\in R$. If $f(x)=x^4+Ax^2+B$ is a square element in $R[x]$, then $A^2-4B=0$.
\end{lemma}
\begin{proof}
  Assume that $f(x)=(x^2+A_1x+B_1)^2$. Comparing the coefficients of all terms on both sides, we have
  $$2A_1=0,\quad 2B_1+A_1^2=A,\quad 2A_1B_1=0,\quad B_1^2=B.$$
  It follows that $A_1=0$, $A=2B_1$ and $B=B_1^2$, which implies that $A^2-4B=0$.
\end{proof}

The character sum is a powerful tool in both theory and application. Below, we recall some facts about character sums of the form $\sum_{a\in\mathbb{F}_q}\psi\big(f(a)\big)$ where $f(x)\in\mathbb{F}_q[x]$ and $\psi$ is a multiplicative character of $\mathbb{F}_q$. Such character sums are called Weil sums.

If $f$ is a quadratic polynomial and $\psi=\eta$, then we can determine the exact value of the associated Weil sum.

\begin{lemma}[{\cite[Theorem 5.48]{lidl1997finite}}]\label{20240825lemma5}
  Let $q$ be an odd prime power and let $f(x)=a_2x^2+a_1x+a_0\in\mathbb{F}_q[x]$ with $a_2\ne 0$. Put $d=a_1^2-4a_0a_2$. Then
  \begin{equation*}
    \sum\limits_{x\in\mathbb{F}_q}\eta\big(f(x)\big)=\begin{cases}
      -\eta(a_2),     & \mbox{if }d\ne 0, \\
      (q-1)\eta(a_2), & \mbox{if }d=0.
    \end{cases}
  \end{equation*}
\end{lemma}

The counting problem in the following Lemma can be addressed using Lemma \ref{20240825lemma5}.

\begin{lemma}[{\cite[Lemma 6.24]{lidl1997finite}}]\label{20240903lemma1}
  Let $q$ be an odd prime power, $b\in\mathbb{F}_q$ and $a_1,a_2\in\mathbb{F}_q^*$. Then
  $$\#\left\{(x_1,x_2)\in\mathbb{F}_{q}^2:\ a_1x_1^2+a_2x_2^2=b\right\}=q+\nu(b)\eta(-a_1a_2),$$
  where the integer-valued function $\nu$ on $\mathbb{F}_q$ is defined by $\nu(b)=-1$ for $b\in\mathbb{F}_q^*$ and $\nu(0)=q-1$.
\end{lemma}

We need the following result of the character sum.

\begin{lemma}\label{20240908lemma2}
  Let $q$ be an odd prime power such that $q\equiv 3\ (\mymod 4)$, then
  $$\sum\limits_{x\in\mathbb{F}_q}\eta(x^4-1)=-1.$$
\end{lemma}

\begin{proof}
  Let $y\in C_0$. Then $y=z^2$ for some $z\in\mathbb{F}_q^*$. Since $\eta(z)\eta(-z)=\eta(-y)=-1$, we have $\eta(z)=-\eta(-z)$. We may assume that $\eta(z)=1$, and then $z=x^2$ for some $x\in\mathbb{F}_q^*$, which implies that $y=x^4$. Hence $C_0=\{x^4:\ x\in\mathbb{F}_q^*\}$. Moreover, for any $y\in C_0$, there exist exactly two elements $x\in\mathbb{F}_q^*$ such that $y=x^4$. It follows that
  \begin{align*}
    \sum\limits_{x\in\mathbb{F}_q}\eta(x^4-1) & =\eta(-1)+2\sum\limits_{y\in C_0}\eta(y-1)    \\
                                              & =\sum\limits_{x\in\mathbb{F}_q}\eta(x^2-1)=-1
  \end{align*}
  by Lemma \ref{20240825lemma5}.
\end{proof}

Jacobsthal sums are also a class of Weil sums that have been extensively studied.

\begin{definition}
  Let $q$ be an odd prime number, $n\in\mathbb{N}_+$ and $a\in\mathbb{F}_q^*$. The sum
  $$H_n(a)=\sum\limits_{x\in\mathbb{F}_q}\eta(x^{n+1}+ax)$$
  is called a Jacobsthal sum.
\end{definition}

\begin{lemma}[{\cite[Theorem 5.52]{lidl1997finite}}]\label{20240906lemma1}
  Let $q$ be an odd prime number, $n\in\mathbb{N}_+$ and $a\in\mathbb{F}_q^*$. We have $H_n(a)=0$ if the largest power of $2$ dividing $q-1$ also divides $n$.
\end{lemma}

\begin{remark}
  If $q\equiv 3\ (\mymod 4)$, then the largest power of $2$ dividing $q-1$ is $2$, which implies that $H_n(a)=0$ for any even $n$ and $a\in\mathbb{F}_q^*$.
\end{remark}

For a general polynomial $f$, we have the following estimate for the associated Weil sum.

\begin{theorem}[{\cite[Theorem 5.41]{lidl1997finite}}]\label{20240904them3}
  Let $q$ be an odd prime power, let $\psi$ be a multiplicative character of $\mathbb{F}_q$ of order $m>1$, and let $f\in\mathbb{F}_q[x]$ be a monic polynomial of positive degree that is not an $m$-th power of a polynomial. Let $d$ be the number of distinct roots of $f$ in its splitting field over $\mathbb{F}_q$. Then for every $a\in\mathbb{F}_q$, we have
  $$\left|\sum\limits_{x\in\mathbb{F}_q}\psi\big(af(x)\big)\right|\le(d-1)\sqrt{q}.$$
\end{theorem}

The following lemma about quadratic character sums of cubic polynomials will also be used later.

\begin{lemma}\label{20240906lemma3}
  Let $q$ be an odd prime power and let $a,b,c,d\in\mathbb{F}_q$ with $a,d\ne 0$. Then
  \begin{align*}
        & \sum\limits_{x\in\mathbb{F}_q}\eta(ax^3+bx^2+cx+d)\eta(x)    \\
    =\  & -\eta(a)+\sum\limits_{x\in\mathbb{F}_q}\eta(dx^3+cx^2+bx+a).
  \end{align*}
\end{lemma}
\begin{proof}
  We have
  \begin{align*}
        & \sum\limits_{x\in\mathbb{F}_q}\eta(ax^3+bx^2+cx+d)\eta(x)              \\
    =\  & \sum\limits_{x\in\mathbb{F}_q^*}\eta(ax^3+bx^2+cx+d)\eta(x)            \\
    =\  & \sum\limits_{x\in\mathbb{F}_q^*}\eta(ax^3+bx^2+cx+d)\eta(x^{-3})       \\
    =\  & \sum\limits_{x\in\mathbb{F}_q^*}\eta(a+bx^{-1}+cx^{-2}+dx^{-3})        \\
    =\  & \sum\limits_{y\in\mathbb{F}_q^*}\eta(a+by+cy^{2}+dy^3)\quad (y=x^{-1}) \\
    =\  & -\eta(a)+\sum\limits_{x\in\mathbb{F}_q}\eta(dx^3+cx^2+bx+a).
  \end{align*}
\end{proof}

From this point until the end of the paper, we assume that $q=p^n$ is an odd prime power such that $q\equiv 3\ (\mymod 4)$, where $p$ is a prime and $n$ is a positive integer. Let $F_{r,u}$ be the function over $\mathbb{F}_q$ defined by (\ref{20240907equation19}).

\begin{lemma}\label{20240825lemma10}
  For any $a\in\mathbb{F}_{q}^*$ and $b\in\mathbb{F}_{q}$, we have $\delta_{F_{r,-u}}(a,b)=\delta_{F_{r,u}}(a,\frac{b}{(-1)^{r+1}})$ and $\beta_{F_{r,-u}}(a,b)=\beta_{F_{r,u}}(a,\frac{b}{(-1)^{r}})$. In particular, $F_{r,u}$ and $F_{r,-u}$ have the same differential and boomerang spectrum.
\end{lemma}

\begin{proof}
  Let $a\in\mathbb{F}_{q}^*$ and $b\in\mathbb{F}_{q}$. By definition, $\delta_{F_{r,-u}}(a,b)$ equals the number of solutions $x\in\mathbb{F}_q$ to the following equation
  $$(x+a)^r\big(1+u\eta(x+a)\big)-x^r\big(1+u\eta(x)\big)=b.$$
  Making the substitution $y=-(x+a)$, we can see that $\delta_{F_{r,-u}}(a,b)$ equals the number of solutions $y\in\mathbb{F}_q$ to the following equation
  $$(y+a)^r\big(1+u\eta(y+a)\big)-y^r\big(1+u\eta(y)\big)=\frac{b}{(-1)^{r+1}}$$
  Hence $\delta_{F_{r,-u}}(a,b)=\delta_{F_{r,u}}(a,\frac{b}{(-1)^{r+1}})$. The assertion for boomerang uniformity can be proved similarly.
\end{proof}

\indent It is known that if $f(x)=x^d$ is a monomial, then $\delta_{f}(a,b)=\delta(1,\frac{b}{a^d})$ and $\beta_{f}(a,b)=\beta_f(1,\frac{b}{a^d})$ for any $a\in\mathbb{F}_{q}^*$ and $b\in\mathbb{F}_{q}$, which implies that $\delta_f=\max_{b\in\mathbb{F}_{q}}\delta_f(1,b)$ and $\beta_f=\max_{b\in\mathbb{F}_{q}^*}\beta_f(1,b)$.  An interesting point is that $F_{r,u}$ has similar properties.

\begin{lemma}\label{20240825lemma7}
  For any $a\in\mathbb{F}_{q}^*$ and $b\in\mathbb{F}_{q}$, we have
  $$\delta_{F_{r,u}}(a,b)=\begin{cases}\delta_{F_{r,u}}(1,\frac{b}{a^r})&\mbox{if}\ \eta(a)=1,\\\delta_{F_{r,u}}(1,\frac{b}{(-1)^{r+1}a^r})&\mbox{if }\eta(a)=-1,\end{cases}$$
  and
  $$\beta_{F_{r,u}}(a,b)=\begin{cases}\beta_{F_{r,u}}(1,\frac{b}{a^r})&\mbox{if}\ \eta(a)=1,\\\beta_{F_{r,u}}(1,\frac{b}{(-1)^{r}a^r})&\mbox{if }\eta(a)=-1.\end{cases}$$
  In particular, we have $\delta_{F_{r,u}}=\max_{b\in\mathbb{F}_{q}}\delta_{F_{r,u}}(1,b)$ and $\beta_{F_{r,u}}=\max_{b\in\mathbb{F}_{q}^*}\beta_{F_{r,u}}(1,b)$.
\end{lemma}

\begin{proof}
  Let $a\in\mathbb{F}_{q}^*$ and $b\in\mathbb{F}_{q}$. By definition, $ \delta_{F_{r,u}}(a,b)$ equals the number of solutions $x\in\mathbb{F}_q$ to the following equation
  $$(x+a)^r\big(1+u\eta(x+a)\big)-x^r\big(1+u\eta(x)\big)=b,$$
  which becomes
  $$(\frac{x}{a}+1)^r\big(1+u\eta(a)\eta(\frac{x}{a}+1)\big)  -(\frac{x}{a})^r\big(1+u\eta(a)\eta(\frac{x}{a})\big)=\frac{b}{a^r}$$
  after dividing both sides by $a^r$. Making the substitution $y=\frac{x}{a}$, we can see that $\delta_{F_{r,u}}$ equals the number of solutions $y\in\mathbb{F}_q$ to the following equation
  $$(y+1)^r\big(1+u\eta(a)\eta(y+1)\big)-y^r\big(1+u\eta(a)\eta(y)\big)=\frac{b}{a^r}.$$
  If $\eta(a)=1$, it is clear that $\delta_{F_{r,u}}(a,b)=\delta_{F_{r,u}}(1,\frac{b}{a^r})$. If $\eta(a)=-1$, making the substitution $z=-(y+1)$, then we can see that $\delta_{F_{r,u}}(a,b)$ equals the number of solutions $z\in\mathbb{F}_q$ to the following equation
  $$(z+1)^r\big(1+u\eta(z+1)\big)-z^r\big(1+u\eta(z)\big)=\frac{b}{(-1)^{r+1}a^r}.$$
  Hence $\delta_{F_{r,u}}(a,b)=\delta_{F_{r,u}}(1,\frac{b}{(-1)^{r+1}a^r})$. The assertion for boomerang uniformity can be proved similarly.
\end{proof}

\section{The Differential Properties of $F_{2,u}$ with $u\in\mathbb{F}_q\setminus\{\pm 1\}$}\label{20240907section2}

In this section, we conduct an initial study of the differential properties of $F_{2,u}$ under the assumption that $u\in\mathbb{F}_q\setminus\{\pm 1\}$.

By Lemma \ref{20240825lemma7}, in order to compute the differential uniformity of $F_{2,u}$, we only need to consider the numbers $\delta_{F_{2,u}}(1,b)$ $(b\in\mathbb{F}_q)$. We have
\begin{align*}
  D_1F_{2,u}(x) & =F_{2,u}(x+1)-F_{2,u}(x)                                \\
                & =(x+1)^2\big(1+u\eta(x+1)\big)-x^2\big(1+u\eta(x)\big).
\end{align*}
Then $D_1F_{2,u}(0)=u+1$ and $D_1F_{2,u}(-1)=u-1$.

Let $\tau_1=\frac{1+u}{u}$ and $\tau_2=\frac{1-u}{u}$. Then $\tau_1+\tau_2=\frac{2}{u}$ and $\tau_1-\tau_2=2$. Since $u\not\in\{\pm1\}$, we have $\tau_i\ne 0$ for $i=1,2$. For any $i,j\in\{0,1\}$ and $b\in\mathbb{F}_q$, we put
$$A_{ij}(b)=\{x\in C_{ij}:\ D_1F_{2,u}(x)=b\}.$$

\noindent{\textbf{Case 1.}}\ If $x\in C_{00}$, then
$$D_1F_{2,u}(x)=(u+1)\big((x+1)^2-x^2\big)=(1+u)(2x+1).$$
The unique possible solution of $D_1F_{2,u}(x)=b$ is $x=\frac{b-(1+u)}{2(1+u)}$. Moreover, we have
\begin{equation}\label{20240825equation4}
  \# A_{00}(b)=\begin{cases}1&\mbox{if}\ \frac{b-(1+u)}{2(1+u)}\in C_{00},\ \mbox{i.e., }\frac{b\pm (1+u)}{2(1+u)}\in C_0,\\0&\mbox{otherwise}.\end{cases}
\end{equation}
\noindent{\textbf{Case 2.}}\ If $x\in C_{11}$, then
$$D_1F_{2,u}(x)=(1-u)\big((x+1)^2-x^2\big)=(1-u)(2x+1).$$
The unique possible solution of $D_1F_{2,u}(x)=b$ is $x=\frac{b-(1-u)}{2(1-u)}$. Moreover, we have
\begin{equation}\label{20240825equation5}
  \# A_{11}(b)=\begin{cases}1&\mbox{if}\ \frac{b-(1-u)}{2(1-u)}\in C_{11},\ \mbox{i.e.,\ }\frac{b\pm (1-u)}{2(1-u)}\in C_1,\\0&\mbox{otherwise}.\end{cases}
\end{equation}
\noindent{\textbf{Case 3.}}\ If $x\in C_{01}$, then
$$D_1F_{2,u}(x)=-2ux^2+2(1-u)x+(1-u).$$
Consider the equation
\begin{equation}\label{20240825equation6}
  D_1F_{2,u}(x)=b\ \Leftrightarrow\ x^2-\tau_2x+\frac{1}{2}(\frac{b}{u}-\tau_2)=0.
\end{equation}
The discriminant of this quadratic equation is $\Delta_{01}=\tau_1\tau_2-2\frac{b}{u}$. Let $x_1$, $x_2$ be the two solutions (possibly equal) of this equation in $\overline{\mathbb{F}_q}$. We have $x_1x_2=\frac{1}{2}(\frac{b}{u}-\tau_2)$ and $(x_1+1)(x_2+1)=\frac{1}{2}(\frac{b}{u}+\tau_1)$. Moreover, $\# A_{01}(b)=2$ if and only if
\begin{align}\label{20240829equation1}
         & \begin{cases}
             \eta(\Delta_{01})=1,   \\
             \eta(x_1)=\eta(x_2)=1, \\
             \eta(x_1+1)=\eta(x_2+1)=-1,
           \end{cases}                                                                                           \notag           \\
  \iff\  & \begin{cases}
             \eta(\tau_1\tau_2-2\frac{b}{u})=1,                       \\
             \eta(\tau_2\pm\sqrt{\tau_1\tau_2-2\frac{b}{u}})=\eta(2), \\
             \eta(\tau_1\pm\sqrt{\tau_1\tau_2-2\frac{b}{u}})=-\eta(2).
           \end{cases}
\end{align}

\noindent{\textbf{Case 4.}}\ If $x\in C_{10}$, then
$$D_1F_{2,u}(x)=2ux^2+2(1+u)x+(1+u).$$
Consider the equation
\begin{equation}\label{20240825equation11}
  D_1F_{2,u}(x)=b\ \Leftrightarrow\ x^2+\tau_1x+\frac{1}{2}(\tau_1-\frac{b}{u})=0.
\end{equation}
The discriminant of this quadratic equation is $\Delta_{10}=\tau_1\tau_2+2\frac{b}{u}$. Let $x_1$, $x_2$ be the two solutions (possibly equal) of this equation in $\overline{\mathbb{F}_q}$. We have $x_1x_2=\frac{1}{2}(\tau_1-\frac{b}{u})$ and $(x_1+1)(x_2+1)=-\frac{1}{2}(\tau_2+\frac{b}{u})$. Moreover, $\# A_{10}(b)=2$ if and only if
\begin{align}\label{20240905equation10}
         & \begin{cases}
             \eta(\Delta_{10})=1,    \\
             \eta(x_1)=\eta(x_2)=-1, \\
             \eta(x_1+1)=\eta(x_2+1)=1.
           \end{cases}\notag                                                                                                \\
  \iff\  & \begin{cases}
             \eta(\tau_1\tau_2+2\frac{b}{u})=1,                         \\
             \eta(-\tau_1\pm\sqrt{\tau_1\tau_2+2\frac{b}{u}})=-\eta(2), \\
             \eta(-\tau_2\pm\sqrt{\tau_1\tau_2+2\frac{b}{u}})=\eta(2).
           \end{cases}
\end{align}

\begin{lemma}\label{20240825lemma9}
  Assume that $\eta(1+u)=\eta(u)$. If $\# A_{10}(b)=2$, then $\# A_{00}(b)=0$.
\end{lemma}
\begin{proof}
  Since $\# A_{10}(b)=2$, we have $x_1,x_2\in C_{10}$, which implies that
  $$1=\eta(x_1)\eta(x_2)=\eta(x_1x_2)=\eta(\frac{(1+u)-b}{2u}).$$
  Since $\eta(1+u)=\eta(u)$, we have $\eta(\frac{b-(1+u)}{2(1+u)})=-1$. By (\ref{20240825equation4}), we have $\# A_{00}(b)=0$.
\end{proof}

\begin{lemma}\label{20240825lemma8}
  Assume that $\eta(1+u)=-\eta(u)$. If $\# A_{01}(b)=2$, then $\# A_{00}(b)=0$.
\end{lemma}

\begin{proof}
  Since $\# A_{01}(b)=2$, we have $x_1,x_2\in C_{01}$, which implies that
  \begin{align*}
    1 & =\eta(x_1+1)\eta(x_2+1)                                 \\
      & =\eta\big((x_1+1)(x_2+1)\big)=\eta(\frac{b+(1+u)}{2u}).
  \end{align*}
  Since $\eta(1+u)=-\eta(u)$, we have $\eta(\frac{b+(1+u)}{2(1+u)})=-1$. By (\ref{20240825equation4}), we have $\# A_{00}(b)=0$.
\end{proof}

\begin{lemma}\label{20240825lemma11}
  Assume that $\eta(1-u)=\eta(u)$. If $\# A_{01}(b)=2$, then $\# A_{11}(b)=0$.
\end{lemma}
\begin{proof}
  Since $\# A_{01}(b)=2$, we have $x_1,x_2\in C_{01}$, which implies that
  $$1=\eta(x_1)\eta(x_2)=\eta(x_1x_2)=\eta(\frac{b-(1-u)}{2u}).$$
  Since $\eta(1-u)=\eta(u)$, we have $\eta(\frac{b-(1-u)}{2(1-u)})=1$. By (\ref{20240825equation5}), we have $\# A_{11}(b)=0$.
\end{proof}

\begin{lemma}\label{20240828lemma2}
  Assume that $\eta(1-u)=-\eta(u)$. If $\# A_{10}(b)=2$, then $\# A_{11}(b)=0$.
\end{lemma}
\begin{proof}
  Since $\# A_{10}(b)=2$, we have $x_1,x_2\in C_{10}$, which implies that
  \begin{align*}
    1 & =\eta(x_1+1)\eta(x_2+1)                                  \\
      & =\eta\big((x_1+1)(x_2+1)\big)=\eta(-\frac{b+(1-u)}{2u}).
  \end{align*}
  Since $\eta(1-u)=-\eta(u)$, we have $\eta(\frac{b+(1-u)}{2(1-u)})=1$. By (\ref{20240825equation5}), we have $\# A_{11}(b)=0$.
\end{proof}

\begin{lemma}\label{20240904lemma3}
  For any $u\in\mathbb{F}_q\setminus\{0,\pm 1\}$, we have $\delta_{F_{2,u}}(1,u\pm 1)\le 4$.
\end{lemma}
\begin{proof}
  We only prove that $\delta_{F_{2,u}}(1,u+1)\le 4$; the proof for $\delta_{F_{2,u}}(1,u-1)\le 4$ is similar. We know that $D_1F_{2,u}(0)=u+1$. Since $\frac{u+1-(1+u)}{2(1+u)}=0\not\in C_0$, by (\ref{20240825equation4}), we have $\# A_{00}(u+1)=0$. Since $\frac{u+1+(1-u)}{2(1-u)}=\frac{1}{1-u}$ and $\frac{u+1-(1-u)}{2(1-u)}=\frac{u}{1-u}$, by (\ref{20240825equation5}), we have
  $$\# A_{11}(u+1)=\begin{cases}1&\mbox{if}\ \eta(1-u)=-1,\ \eta(u)=1,\\0&\mbox{otherwise}.\end{cases}$$
  Note that $\Delta_{10}=\tau_1^2$, which implies that the two solutions of the equation (\ref{20240825equation11}) are $0$ and $-\tau_1$. Hence
  $$\# A_{10}(u+1)=\begin{cases}1&\mbox{if}\ \eta(1+u)=\eta(u)=-1,\\0&\mbox{otherwise}.\end{cases}$$
  It follows that $\# A_{11}(u+1)+\# A_{10}(u+1)\le 1$ and thus $\delta_{F_{2,u}}(1,u+1)\le 4$.
\end{proof}

\begin{corollary}\label{20240828corol1}
  For any $u\in\mathbb{F}_q\setminus\{0,\pm 1\}$, we have $\delta_{F_{2,u}}\le 5$. Moreover, we have following conclusions:
  \begin{enumerate}
    \item if $\eta(1+u)=\eta(1-u)$, then $\delta_{F_{2,u}}\le 4$.
    \item if $\eta(1+u)=\eta(u-1)=\eta(u)$, then for any $b\in\mathbb{F}_q$, $\delta_{F_{2,u}}(1,b)=5$ if and only if $\# A_{00}(b)=\# A_{11}(b)=\# A_{10}(b)=1$ and $\# A_{01}(b)=2$.
    \item if $\eta(1+u)=\eta(u-1)=-\eta(u)$, then for any $b\in\mathbb{F}_q$, $\delta_{F_{2,u}}(1,b)=5$ if and only if $\# A_{00}(b)=\# A_{11}(b)=\# A_{01}(b)=1$ and $\# A_{10}(b)=2$.
  \end{enumerate}
\end{corollary}
\begin{proof}
  The first assertion follows immediately from Lemma \ref{20240825lemma9}, Lemma \ref{20240825lemma8} and Lemma \ref{20240904lemma3}. It is clear that for any $b\in\mathbb{F}_q$, $\delta_{F_{2,u}}(1,b)=5$ if and only if one of the following conditions holds:
  \begin{enumerate}[label=\roman*)]
    \item $\# A_{00}(b)=\# A_{11}(b)=\# A_{10}(b)=1$, $\# A_{01}(b)=2$;
    \item $\# A_{00}(b)=\# A_{11}(b)=\# A_{01}(b)=1$, $\# A_{10}(b)=2$;
    \item $\# A_{00}(b)=1$,\ $\# A_{11}(b)=0$,\ $\# A_{01}(b)=\# A_{10}(b)=2$;
    \item $\# A_{00}(b)=0$,\ $\# A_{11}(b)=1$,\ $\# A_{01}(b)=\# A_{10}(b)=2$.
  \end{enumerate}
  \begin{enumerate}
    \item Assume that $\eta(1+u)=\eta(1-u)=\eta(u)$. By Lemma \ref{20240825lemma9}, neither of the condition ii) and the condition iii) can hold. By Lemma \ref{20240825lemma11}, neither of the condition i) and the condition iv) can hold. Hence none of the conditions i)-iv) can hold, which implies that $\delta_{F_{2,u}}\le 4$.\\
          \indent\quad Assume that $\eta(1+u)=\eta(1-u)=-\eta(u)$. By Lemma,\ref{20240825lemma8}, neither of the conditions i) and iii) can hold. By Lemma \ref{20240828lemma2}, neither of the conditions ii) and iv) can hold. Hence, none of the conditions i)-iv) can hold, which implies that $\delta_{F_{2,u}}\le 4$.
    \item  By Lemma \ref{20240825lemma9}, neither of the conditions ii) and iii) can hold. By Lemma \ref{20240828lemma2}, neither of the conditions ii) and iv) can hold. Hence only the condition i) can hold.
    \item By Lemma \ref{20240825lemma8}, neither of the conditions i) and iii) can hold. By Lemma \ref{20240825lemma11}, neither of the conditions i) and iv) can hold. Hence, only condition ii) can hold.
  \end{enumerate}
\end{proof}

\section{The Differential Uniformity of $F_{2,u}$ with $u\in\mathbb{F}_q\setminus\mathcal{U}$}\label{20240907section3}

In this section, we determine the differential uniformity of $F_{2,u}$ for any $u\in\mathbb{F}_q\setminus\mathcal{U}$. For the sake of notation simplicity, we use $[m]$ to denote the set $\{1,\cdots,m\}$ for any positive integer $m$.

\begin{theorem}\label{20240905them2}
  If $q\ge 27535^2$, then for any $u\in\mathbb{F}_q\setminus\mathcal{U}$ with $\eta(1+u)=\eta(u-1)$, we have $\delta_{F_{2,u}}=5$.
\end{theorem}

\begin{proof}
  We only prove this theorem for the case where $\eta(1+u)=\eta(u-1)=\eta(u)$; the proof for the case where $\eta(1+u)=\eta(u-1)=-\eta(u)$ is similar. By 2) of Corollary \ref{20240828corol1}, $\delta_{F_{2,u}}=5$ if and only if $\# A_{00}(b)=\# A_{11}(b)=\# A_{10}(b)=1$ and $\# A_{01}(b)=2$.

  By (\ref{20240825equation4}), we have $\# A_{00}(b)=1$ if and only if $\eta(\frac{b\pm(1+u)}{2(1+u)})=1$, i.e., $\eta(\frac{b}{u}\pm\tau_1)=\eta(2)$ noticing that $\eta(1+u)=\eta(u)$. Similarly, we have $\# A_{11}(b)=1$ if and only if $\eta(\frac{b}{u}\pm\tau_2)=\eta(2)$. By (\ref{20240829equation1}), we have $\# A_{01}(b)=2$ if and only if
  \begin{equation*}
    \begin{cases}
      \eta(\tau_1\tau_2-2\frac{b}{u})=1,                       \\
      \eta(\tau_2\pm\sqrt{\tau_1\tau_2-2\frac{b}{u}})=\eta(2), \\
      \eta(\tau_1\pm\sqrt{\tau_1\tau_2-2\frac{b}{u}})=-\eta(2).
    \end{cases}
  \end{equation*}
  Assume that $\eta(\Delta_{10})=\eta(\tau_1\tau_2+2\frac{b}{u})=1$ and let $x_1$, $x_2$ be the two solutions of equation (\ref{20240825equation11}) in $\mathbb{F}_q$. Note that $x_1x_2=\frac{1}{2}(\tau_1-\frac{b}{u})$. If $\# A_{00}(b)=1$, then $\eta(x_1x_2)=-1$, which implies that $\eta(x_1)=-\eta(x_2)$. Hence under the assumption $\# A_{00}(b)=1$, we have $\# A_{10}(b)=1$ if
  \begin{equation*}
    \begin{cases}
      \eta(\Delta_{10})=1\iff\eta(\tau_1\tau_2+2\frac{b}{u})=1,                    \\
      \eta(-\tau_2+y)=\eta(2), \mbox{where }y\ \mbox{is the (only) square root of} \\
      \qquad\qquad\qquad\tau_1\tau_2+2\frac{b}{u}\mbox{ such that }\eta(-\tau_1+y)=-\eta(2).
    \end{cases}
  \end{equation*}
  Now we prove that if $q$ is sufficiently large, then there exists $b\in\mathbb{F}_q$ such that
  \begin{equation}\label{20240905equation7}
    \begin{cases}
      \eta(\frac{b}{u}\pm\tau_1)=\eta(2),                                   \\
      \eta(\frac{b}{u}\pm\tau_2)=\eta(2),                                   \\
      \eta(\tau_1\tau_2\pm 2\frac{b}{u})=1,                                 \\
      \eta(\tau_2\pm\sqrt{\tau_1\tau_2-2\frac{b}{u}})=\eta(2),              \\
      \eta(\tau_1\pm\sqrt{\tau_1\tau_2-2\frac{b}{u}})=-\eta(2),             \\
      \eta(-\tau_2+y)=\eta(2), \mbox{where }y\ \mbox{is the square root of} \\
      \qquad\quad\tau_1\tau_2+2\frac{b}{u}\mbox{ such that }\eta(-\tau_1+y)=-\eta(2).
    \end{cases}
  \end{equation}
  We use $N(u)$ to denote the number of all $b\in\mathbb{F}_q$ satisfying these conditions. Making the substitutions $y^2=\tau_1\tau_2+2\frac{b}{u}$ and $z^2=\tau_1\tau_2-2\frac{b}{u}$, we have $y^2+z^2=2\tau_1\tau_2$, $\frac{b}{u}=\frac{y^2-\tau_1\tau_2}{2}$ and $N(u)$ equals
  \begin{align*}
        & \frac{1}{2}\cdot\#\left\{(y,z)\in{\mathbb{F}_q^*}^2:\ \begin{cases}
                                                                  y^2+z^2=2\tau_1\tau_2,                             \\
                                                                  \eta(\frac{y^2-\tau_1\tau_2}{2}\pm\tau_1)=\eta(2), \\
                                                                  \eta(\frac{y^2-\tau_1\tau_2}{2}\pm\tau_2)=\eta(2), \\
                                                                  \eta(\tau_2\pm z)=\eta(2),                         \\
                                                                  \eta(\tau_1\pm z)=-\eta(2),                        \\
                                                                  \eta(-\tau_2+y)=\eta(2),                           \\
                                                                  \eta(-\tau_1+y)=-\eta(2)
                                                                \end{cases}\right\}\notag \\
    =\  & \frac{1}{2}\cdot\#\left\{(y,z)\in{\mathbb{F}_q^*}^2:\ \begin{cases}
                                                                  y^2+z^2=2\tau_1\tau_2,            \\
                                                                  \eta(y^2-\tau_1^2)=1,             \\
                                                                  \eta(y^2-\frac{1-3u}{u}\tau_1)=1, \\
                                                                  \eta(y^2-\tau_2^2)=1,             \\
                                                                  \eta(y^2-\frac{1+3u}{u}\tau_2)=1, \\
                                                                  \eta(\tau_2\pm z)=\eta(2),        \\
                                                                  \eta(\tau_1\pm z)=-\eta(2),       \\
                                                                  \eta(-\tau_2+y)=\eta(2),          \\
                                                                  \eta(-\tau_1+y)=-\eta(2)
                                                                \end{cases}\right\}\notag                  \\
    =\  & \frac{1}{2}\cdot\#\left\{(y,z)\in{\mathbb{F}_q^*}^2:\ \begin{cases}
                                                                  y^2+z^2=2\tau_1\tau_2,            \\
                                                                  \eta(y^2-\frac{1-3u}{u}\tau_1)=1, \\
                                                                  \eta(y^2-\frac{1+3u}{u}\tau_2)=1, \\
                                                                  \eta(\tau_2\pm z)=\eta(2),        \\
                                                                  \eta(\tau_1\pm z)=-\eta(2),       \\
                                                                  \eta(y\pm\tau_1)=-\eta(2),        \\
                                                                  \eta(y\pm\tau_2)=\eta(2)
                                                                \end{cases}\right\}.
  \end{align*}
  Put
  \begin{align*}
     & p_1(y,z)=-2(y+\tau_1),             &
     & p_2(y,z)=-2(y-\tau_1),               \\
     & p_3(y,z)=2(y+\tau_2),              &
     & p_4(y,z)=2(y-\tau_2),                \\
     & p_5(y,z)=y^2-\frac{1-3u}{u}\tau_1, &
     & p_6(y,z)=y^2-\frac{1+3u}{u}\tau_2,   \\
     & p_7(y,z)=-2(z+\tau_1),             &
     & p_8(y,z)=-2(\tau_1-z),               \\
     & p_9(y,z)=2(z+\tau_2),              &
     & p_{10}(y,z)=2(\tau_2-z).
  \end{align*}
  Note that since $u\ne 0$, none of $\pm\tau_1$ or $\pm\tau_2$ is a root of $p_5$ or $p_6$ (viewed as polynomials of $y$). Moreover, since $u\not\in\{\pm\frac{1}{3}\}$ when $p\ne 3$, neither $p_5$ nor $p_6$ is the square of some polynomial. Then
  $N(u)$ equals
  \begin{align*}
        & \frac{1}{2}\cdot\#\left\{(y,z)\in{\mathbb{F}_q^*}^2:\ \begin{cases}y^2+z^2=2\tau_1\tau_2,\\ \eta(p_i(y,z))=1\ \mbox{for any}\ i\in[10]\end{cases}\right\}                                                                       \\
    =\  & \frac{1}{2^{11}}\cdot\Bigg(\sum\limits_{\substack{y,z\in\mathbb{F}_q                                                                                                                                                            \\y^2+z^2=2\tau_1\tau_2}}\prod\limits_{i=1}^{10}\Big(1+\eta\big(p_i(y,z)\big)\Big)\\
        & \qquad\qquad\quad\quad -\sum\limits_{\substack{(y,z)\in A                                                                                   }}\prod\limits_{i=1}^{10}\Big(1+\eta\big(p_i(y,z)\big)\Big)\Bigg)                   \\
    =\  & \frac{1}{2^{11}}\Bigg(\sum\limits_{I}S_I-\sum\limits_{\substack{(y,z)\in A                                                                                   }}\prod\limits_{i=1}^{10}\Big(1+\eta\big(p_i(y,z)\big)\Big)\Bigg),
  \end{align*}
  where $A=\big\{(y,z)\in\mathbb{F}_q^2:\ y^2+z^2=2\tau_1\tau_2,\ \mbox{and\ }yz=0\ \mbox{or}\ p_i(y,z)=0\ \mbox{for some}\ i\in[10]\big\}$, $I$ runs over all subsets of $[10]$, and
  $$S_I=\sum\limits_{\substack{y,z\in\mathbb{F}_q\\y^2+z^2=2\tau_1\tau_2}}\eta\Big(\displaystyle\prod_{i\in I}p_i(y,z)\Big).$$
  It is not difficult to see that
  \begin{align*}
    A\subset & \Bigg\{(0,\pm\sqrt{2\tau_1\tau_2}),\ (\pm\tau_1,\pm\sqrt{\frac{1-3u}{u}\tau_1}),                      \\
             & \ \ \ (\pm\sqrt{2\tau_1\tau_2},0),\ (\pm\tau_2,\pm\sqrt{\frac{1+3u}{u}\tau_2}),                       \\
             & \ \ \ (\pm\sqrt{\frac{1-3u}{u}\tau_1},\pm\tau_1),\ (\pm\sqrt{\frac{1+3u}{u}\tau_2},\pm\tau_2)\Bigg\},
  \end{align*}
  which implies that $\# A\le 20$. It follows that
  \begin{equation}\label{20240905equation6}
    N(u)\ge\frac{1}{2^{11}}(\sum\limits_{I}S_I-5\cdot 2^{12})
  \end{equation}
  To prove that $N(u)>0$, it remains to estimate $S_I$ for any $I\subset[10]$.
  If $I=\emptyset$, then by Lemma \ref{20240903lemma1}, we have
  \begin{equation}\label{20240904equation2}
    S_{I}=\#\left\{(y,z)\in\mathbb{F}_q^2:\ y^2+z^2=2\tau_1\tau_2\right\}=q+1.
  \end{equation}
  If $I=I^{(1)}:=\{5,7,8\}$, then
  \begin{align*}
    S_I & =\sum\limits_{\substack{y,z\in\mathbb{F}_q                                \\y^2+z^2=2\tau_1\tau_2}}\eta\Big((y^2-\frac{1-3u}{u}\tau_1)(\tau_1^2-z^2)\Big)\\
        & =\sum\limits_{\substack{y,z\in\mathbb{F}_q                                \\y^2+z^2=2\tau_1\tau_2}}\eta\Big((y^2-\frac{1-3u}{u}\tau_1)(\tau_1^2-2\tau_1\tau_2+y^2)\Big)\\
        & =\sum\limits_{\substack{y,z\in\mathbb{F}_q                                \\y^2+z^2=2\tau_1\tau_2}}\eta\Big((y^2-\frac{1-3u}{u}\tau_1)(y^2-\frac{1-3u}{u}\tau_1)\Big)\\
        & \ge\#\left\{(y,z)\in\mathbb{F}_q^2:\ y^2+z^2=2\tau_1\tau_2\right\}-4=q-3.
  \end{align*}
  Similarly, if $I=I^{(2)}:=\{6,9,10\}$, then $S_I\ge q-3$; if $I=I^{(3)}:=\{5,6,7,8,9,10\}$, then $S_I\ge q-7$.

  Now assume that $\# I\ge 1$ and $I\ne I^{(i)}$ for any $1\le i\le 3$. We can divide $I$ into two parts: $I=I_1\cup I_2$, where $I_1\subset [6]$ and $I_2\subset\{7,8,9,10\}$.
  \begin{enumerate}
    \item If $\# I_2=0$, then $\prod_{i\in I}p_i(y,z)$ is a polynomial of $y$.
    \item If $\# I_2=1$, then $\prod_{i\in I}p_i(y,z)$ has the form $\phi(y)(z+a)$, where $\phi\in\mathbb{F}_q[x]$ and $a\in\{\pm\tau_1,\pm\tau_2\}$.
    \item If $\# I_2=3$, then using the relation $z^2=2\tau_1\tau_2-y^2$, $\prod_{i\in I}p_i(y,z)$ can be transformed into the form $\phi(y)(z+a)$, where $\phi\in\mathbb{F}_q[x]$ and $a\in\{\pm\tau_1,\pm\tau_2\}$.
    \item If $\# I_2=4$, then using the relation $z^2=2\tau_1\tau_2-y^2$, $\prod_{i\in I}p_i(y,z)$ can be transformed into a polynomial of $y$.
    \item Now assume that $\# I_2=2$.
          \begin{enumerate}
            \item If $I_2=\{7,8\}$ or $\{9,10\}$, since $(z+\tau_1)(z-\tau_1)=z^2-\tau_1^2=2\tau_1\tau_2-y^2-\tau_1^2=-(y^2-\frac{1+3u}{u}\tau_1)$ and $(z+\tau_2)(z-\tau_2)=z^2-\tau_2^2=2\tau_1\tau_2-y^2-\tau_2^2=-(y^2-\frac{1-3u}{3}\tau_2)$, $\prod_{i\in I}p_i(y,z)$ can be transformed into a polynomial of $y$.
            \item If $I_2\ne\{7,8\}$ and $\{9,10\}$, then we claim that $\prod_{i\in I}p_i(y,z)$ can be transformed into the form $\phi(y)(z+dy^2+a)$, where $\phi\in\mathbb{F}_q[x]$ and $d,a\in\mathbb{F}_q^*$ satisfy that $4ad+8d^2\tau_1\tau_2+1\ne 0$. We take $I_2=\{7,9\}$ as an example. Indeed, if $I_2=\{7,9\}$, since $(z+\tau_1)(z+\tau_2)=z^2+(\tau_1+\tau_2)z+\tau_1\tau_2=(\tau_1+\tau_2)z+3\tau_1\tau_2-y^2$, $\prod_{i\in I}p_i(y,z)$ can be transformed into the form $\phi(y)(z+dy^2+a)$, where $d=-\frac{1}{\tau_1+\tau_2}$ and $a=\frac{3\tau_1\tau_2}{\tau_1+\tau_2}$. It is easy to verify that
                  $$4ad+8d^2\tau_1\tau_2+1=\big(\frac{\tau_1-\tau_2}{\tau_1+\tau_2}\big)^2\ne0.$$
          \end{enumerate}
  \end{enumerate}
  In summary, under the condition that $y^2+z^2=2\tau_1\tau_2$, $\prod_{i\in I}p_i(y,z)$ can be transformed into one of the following forms:
  \begin{enumerate}[label=\Roman*)]
    \item $\gamma(y)$, where $\gamma\in\mathbb{F}_q[x]$;
    \item $\phi(y)(z+a)$, where $\phi\in\mathbb{F}_q[x]$ and $a\in\{\pm\tau_1,\pm\tau_2\}$;
    \item $\phi(y)(z+dy^2+a)$, where $\phi\in\mathbb{F}_q[x]$, $d,a\in\mathbb{F}_q^*$ satisfy that $4ad+8d^2\tau_1\tau_2+1\ne 0$.
  \end{enumerate}
  The last two cases are collectively denoted as $\phi(y)\big(z+\rho(y)\big)$. Note that
  \begin{align}\label{20240904equation1}
          & \sum\limits_{\substack{y,z\in\mathbb{F}_q                                                                                                                              \\y^2+z^2=2\tau_1\tau_2}}\Bigg(1+\eta\Big(\phi(y)\big(z+\rho(y)\big)\Big)\Bigg)\notag\\
    =\    & \#\left\{(y,z,t)\in\mathbb{F}_q^3:\ \begin{cases}y^2+z^2=2\tau_1\tau_2,\\ \phi(y)\big(z+\rho(y)\big)=t^2\end{cases}\right\}\notag                                      \\
    =\    & \#\left\{(y,z)\in\mathbb{F}_q^2:\ y^2+z^2=2\tau_1\tau_2,\ \phi(y)=0\right\}\notag                                                                                      \\
          & +\#\left\{(y,z,t)\in\mathbb{F}_q^3:\ \begin{cases}y^2+z^2=2\tau_1\tau_2,         \\
                                                   z=\frac{t^2}{\phi(y)}-\rho(y), \\ \phi(y)\ne 0\end{cases}\right\}\notag                                                           \\
    \ge\  & \#\left\{(y,z,t)\in\mathbb{F}_q^3:\ \begin{cases}y^2+z^2=2\tau_1\tau_2,\\ z=\frac{t^2}{\phi(y)}-\rho(y),\\ \phi(y)\ne 0\end{cases}\right\}\notag                       \\
    =\    & \#\left\{(y,t)\in\mathbb{F}_q^2:\ \begin{cases}y^2+\big(\frac{t^2}{\phi(y)}-\rho(y)\big)^2=2\tau_1\tau_2,\\ \phi(y)\ne 0\end{cases}\right\} \notag \\
    \ge\  & \# A_{\Omega}(\mathbb{F}_q)-\deg(\phi),
  \end{align}
  where
  $$\Omega(t,y)=t^4-2\phi(y)\rho(y)t^2+\phi(y)^2\big(\rho(y)^2+y^2-2\tau_1\tau_2\big).$$
  We claim that $\Omega(t,y)$ is absolutely irreducible. Firstly, we have
  \begin{align*}
        & \big(2\phi(y)\rho(y)\big)^2-4\phi(y)^2\big(\rho(y)^2+y^2-2\tau_1\tau_2\big) \\
    =\  & -4\phi(y)^2(y^2-2\tau_1\tau_2),
  \end{align*}
  which is not a square element in $\overline{\mathbb{F}_q}[y]$ since $2\tau_1\tau_2\ne 0$. If $\rho(y)=a$ with $a\in\{\pm\tau_1,\pm\tau_2\}$, then $\rho(y)^2-2\tau_1\tau_2\ne 0$ since $u\not\in\{\pm\frac{1}{3}\}$ when $p\ne 3$, which implies that $\rho(y)^2+y^2-2\tau_1\tau_2$ is not a square element in $\overline{\mathbb{F}_q}[y]$. If $\rho(y)=dy^2+a$ with $d,a\in\mathbb{F}_q^*$ and $4ad+8d^2\tau_1\tau_2+1\ne 0$, then $\rho(y)^2+y^2-2\tau_1\tau_2$ is also not a square element in $\overline{\mathbb{F}_q}[y]$. Indeed, in this case, we have
  $$\rho(y)^2+y^2-2\tau_1\tau_2=d^2(y^4+\frac{2da+1}{d^2}y^2+\frac{a^2-2\tau_1\tau_2}{d^2}).$$
  By Lemma \ref{20240904lemma1}, if it is a square element in $\overline{\mathbb{F}_q}[y]$, then
  $$(\frac{2da+1}{d^2})^2-4\frac{a^2-2\tau_1\tau_2}{d^2}=0\iff 4da+8d^2\tau_1\tau_2+1=0,$$
  which is a contradiction. Hence $\rho(y)^2+y^2-2\tau_1\tau_2$ is also not a square element in $\overline{\mathbb{F}_q}[y]$. By Lemma \ref{20240904lemma2}, $\Omega(t,y)$ is absolutely irreducible. By Theorem \ref{20240904them1}, we have
  \begin{align*}
    \left|\# A_{\Omega}(\mathbb{F}_q)-q\right|
     & \le\big(\deg(\Omega)-1\big)\big(\deg(\Omega)-2\big)\sqrt{q} \\
     & \qquad\qquad+5\deg(\Omega)^{\frac{13}{3}}.
  \end{align*}
  Then by (\ref{20240904equation1}), we have
  \begin{align}\label{20240904equation3}
    S_I & =\sum\limits_{\substack{y,z\in\mathbb{F}_q                                          \\y^2+z^2=2\tau_1\tau_2}}\eta\Big(\phi(y)\big(z+\rho(y)\big)\Big)\notag\\
        & \ge\# A_{\Omega}(\mathbb{F}_q)-\deg(\phi)-\sum\limits_{\substack{y,z\in\mathbb{F}_q \\y^2+z^2=2\tau_1\tau_2}}1\\
        & \ge-\big(\deg(\Omega)-1\big)\big(\deg(\Omega)-2\big)\sqrt{q} \notag                 \\
        & \qquad\qquad\qquad -5\deg(\Omega)^{\frac{13}{3}}-\deg(\phi)-1.\notag
  \end{align}
  By definition, it is easy to see that
  $$\deg(\Omega)=\begin{cases}
      \max\{4,2+2\deg(\phi)\} & \mbox{if }\deg(\rho)=0, \\
      4+2\deg(\phi)           & \mbox{if }\deg(\rho)=2.
    \end{cases}$$
  Finally, we address the case I), i.e., $\prod_{i\in I}p_i(y,z)=\gamma(y)$ with $\gamma\in\mathbb{F}_q[x]$. Note that
  \begin{align*}
    S_I & =\sum\limits_{\substack{y,z\in\mathbb{F}_q                                                                                       \\y^2+z^2=2\tau_1\tau_2}}\eta\Big(\gamma(y)\Big)\\
        & =\sum\limits_{y\in\mathbb{F}_q}\eta\big(\gamma(y)\big)\Big(1+\eta(2\tau_1\tau_2-y^2)\Big)                                        \\
        & =\sum\limits_{y\in\mathbb{F}_q}\eta\big(\gamma(y)\big)+\sum\limits_{y\in\mathbb{F}_q}\eta\big(\gamma(y)(2\tau_1\tau_2-y^2)\big).
  \end{align*}
  It is easy to see that since $u\not\in\{\pm\frac{1}{3}\}$ when $p\ne 3$, none of $\pm\tau_1$, $\pm\tau_2$ is a root of $2\tau_1\tau_2-y^2$. Then from the above discussion, we know that neither $\gamma(y)$ nor $\gamma(y)(2\tau_1\tau_2-y^2)$ is a square element in $\overline{\mathbb{F}_q}[y]$ (note that $I\ne I^{(i)}$ for any $1\le i\le 3$). By Theorem \ref{20240904them3}, we have
  \begin{align}\label{20240904equation5}
    S_I & \ge\big(1-{\rm{d}}(\gamma)\big)\sqrt{q}+\Big(1-\big(2+{\rm{d}}(\gamma)\big)\Big)\sqrt{q}\notag \\
        & =-2{\rm{d}}(\gamma)\sqrt{q},
  \end{align}
  where ${\rm{d}}(\gamma)$ is the number of distinct roots of $\gamma$ in the its splitting field over $\mathbb{F}_q$. In summary, for the case II) and the case III), we can use (\ref{20240904equation3}) to estimate $S_I$; for the case I), we can use (\ref{20240904equation5}) to estimate $S_I$. Eventually, we will obtain a lower bound for $\sum_IS_I$ of the form $(4q-12)+m_1\sqrt{q}+m_2$, where $m_1\in\mathbb{Z}$ and $m_2\in\mathbb{R}$. Due to the large number of terms and the complexity of the explanation, we use a Python program to calculate the values of $m_1$ and $m_2$, which can be found in the appendix. It turns out that we can take $m_1=-98312$ and $m_2=-325643353$. By (\ref{20240905equation6}), we have
  \begin{align*}
    N(u) & \ge\frac{1}{2^{11}}(\sum\limits_{I}S_I-5\cdot 2^{12}) \\
         & \ge\frac{1}{2^{11}}(4q-98312\sqrt{q}-325663845).
  \end{align*}
  It is easy to see that if $q\ge 27535^2$, then $N(u)>0$.
\end{proof}

\begin{remark}
  Since the number $27535^2$ is very large, it is impossible to exhaustively check all cases for $q<27535^2$ in a short time. Numerical results suggest that the theorem is true when $q\ge 4027$.
\end{remark}

\begin{theorem}
  If $q\ge 27535^2$, then for any $u\in\mathbb{F}_q\setminus\mathcal{U}$ with $\eta(1+u)=\eta(1-u)$, we have $\delta_{F_{2,u}}=4$.
\end{theorem}
\begin{proof}
  If $\eta(1+u)=\eta(1-u)=-\eta(u)$, then $u'=-u$ satisfies that $\eta(1+u')=\eta(1-u')=\eta(u')$. Since by Lemma \ref{20240825lemma10}, $F_{2,u}$ and $F_{2,u'}$ have the same differential spectrum, we only need to prove this theorem for the case where $\eta(1+u)=\eta(1-u)=\eta(u)$. By 1) of Corollary \ref{20240828corol1}, it suffices to show that there exists $b\in\mathbb{F}_q$ such that $\delta_{F_{2,u}}(1,b)=4$. Now we prove that if $q$ is sufficiently large, then there exists $b\in\mathbb{F}_q$ such that
  $\# A_{01}(b)=2$ and $\# A_{00}(b)=\# A_{10}(b)=1$.

  From the proof of Theorem \ref{20240905them2}, we can see that it suffices to prove that if $q$ is sufficiently large, then there exists $b\in\mathbb{F}_q$ such that
  \begin{equation}\label{20240905equation8}
    \begin{cases}
      \eta(\frac{b}{u}\pm\tau_1)=\eta(2),                                \\
      \eta(\tau_1\tau_2\pm 2\frac{b}{u})=1,                              \\
      \eta(\tau_2\pm\sqrt{\tau_1\tau_2-2\frac{b}{u}})=\eta(2),           \\
      \eta(\tau_1\pm\sqrt{\tau_1\tau_2-2\frac{b}{u}})=-\eta(2),          \\
      \eta(-\tau_2+y)=\eta(2), \mbox{where }y\ \mbox{is the square root} \\
      \quad\mbox{of }\tau_1\tau_2+2\frac{b}{u}\mbox{ such that }\eta(-\tau_1+y)=-\eta(2).
    \end{cases}
  \end{equation}
  These conditions are obtained by removing the conditions $\eta(\frac{b}{u}\pm\tau_2)=\eta(2)$ from the conditions (\ref{20240905equation7}). Note that in the proof of Theorem \ref{20240905them2}, after obtaining the conditions (\ref{20240905equation7}), we no longer use the condition $\eta(u-1)=\eta(u)$. Hence by Theorem \ref{20240905them2}, if $q\ge 27535^2$, then there exists $b\in\mathbb{F}_q$ satisfying the conditions (\ref{20240905equation8}).
\end{proof}

\begin{remark}
  Here, for brevity, we directly use the bound $27535^2$ obtained in Theorem \ref{20240905them2}. However, readers can perform similar estimations as in the proof of Theorem \ref{20240905them2} to obtain a smaller bound. This will reduce the computational resources required for brute-force verification. Numerical results suggest that the theorem is true when $q\ge 839$.
\end{remark}

\section{The Differential Uniformity of $F_{2,\pm\frac{1}{3}}$ when $p\ne 3$}\label{20240907section4}

In this section, we determine the differential uniformity of $F_{2,\pm\frac{1}{3}}$ when $p\ne 3$. By Lemma \ref{20240825lemma10},  $F_{2,\frac{1}{3}}$ and $F_{2,-\frac{1}{3}}$ have the same differential spectrum. Hence we only need to study $F_{2,\frac{1}{3}}$. We have $\tau_1=4$, $\tau_2=2$, $D_1F_{2,\frac{1}{3}}(0)=\frac{4}{3}$ and $D_1F_{2,\frac{1}{3}}(-1)=-\frac{2}{3}$.

By (\ref{20240825equation4}), we have
\begin{equation}\label{20240825equation12}
  \# A_{00}(b)=\begin{cases}1&\mbox{if}\ 2(3b\pm 4)\in C_{0},\\0&\mbox{otherwise}.\end{cases}
\end{equation}
By (\ref{20240825equation5}), we have
\begin{equation}\label{20240905equation9}
  \# A_{11}(b)=\begin{cases}1&\mbox{if}\ 3b\pm 2\in C_{1},\\0&\mbox{otherwise}.\end{cases}
\end{equation}
The equation (\ref{20240825equation6}) becomes
\begin{equation}\label{20240905equation11}
  x^2-2x+\frac{3b-2}{2}=0,
\end{equation}
with $\Delta_{01}=2(4-3b)$, $x_1x_2=\frac{3b-2}{2}$ and $(x_1+1)(x_2+1)=\frac{3b+4}{2}$. The equation (\ref{20240825equation11}) becomes
\begin{equation}\label{20240905equation12}
  x^2+4x+\frac{4-3b}{2}=0,
\end{equation}with $\Delta_{10}=2(3b+4)$, $x_1x_2=\frac{4-3b}{2}$ and $(x_1+1)(x_2+1)=-\frac{3b+2}{2}$.

Note that $\eta(1-\frac{1}{3})=\eta(1+\frac{1}{3})$ if and only if $\eta(2)=1$, which is furthermore equivalent to saying that $q\equiv 7\ (\mymod\ 8)$.

\begin{lemma}\label{20240824lemma1}
  For any $b\in\mathbb{F}_q$, $\# A_{00}(b)$ and $\# A_{01}(b)$ cannot both be non-zero.
\end{lemma}
\begin{proof}
  If $\# A_{01}(b)>0$, then $\Delta_{01}=2(4-3b)$ is a square element in $\mathbb{F}_{q}$, which implies that $\eta\big(2(3b-4)\big)=-1$ or $0$, which implies that $\# A_{00}(b)=0$ by (\ref{20240825equation12}).
\end{proof}

\begin{corollary}\label{20240925corol1}
  We have $\delta_{F_{2,\frac{1}{3}}}\le 4$.
\end{corollary}
\begin{proof}
  If $\eta(2)=1$, then $\eta(1+\frac{1}{3})=\eta(1-\frac{1}{3})$. By 1) of Corollary \ref{20240828corol1}, we have $\delta_{F_{2,\frac{1}{3}}}\le 4$. If $\eta(2)=-1$, then $\eta(1+\frac{1}{3})=\eta(\frac{1}{3}-1)=\eta(\frac{1}{3})$. By Lemma \ref{20240824lemma1} and 2) of Corollary \ref{20240828corol1}, we have $\delta_{F_{2,\frac{1}{3}}}\le 4$.
\end{proof}

We first consider the case where $q\equiv 7\ (\mymod\ 8)$.

\begin{lemma}\label{20240825lemma12}
  If $q\equiv 7\ (\mymod\ 8)$, then for any $b\in\mathbb{F}_q$, $\#A_{01}(b)$ and $\#A_{10}(b)$ cannot both be  non-zero.
\end{lemma}
\begin{proof}
  Note that $A_{01}(\frac{4}{3})=\{x\in C_{01}:\ x^2-2x+1=0\}\subset\{1\}$. Since $\eta(1+1)=\eta(2)=1$, we have $A_{01}(\frac{4}{3})=\emptyset$. Note that $A_{10}(-\frac{4}{3})=\{x\in C_{10}:\ x^2+4x+4=0\}\subset\{-2\}$. Since $\eta(-2+1)=\eta(-1)=-1$, we have $A_{10}(-\frac{4}{3})=\emptyset$.

  \indent Now assume that $b\not\in\{\pm\frac{4}{3}\}$. Assume, for a contradiction, that $x_1\in A_{01}(b)$ and $x_2\in A_{10}(b)$. Then both $\frac{\Delta_{01}}{4}=\frac{4-3b}{2}$ and $\frac{\Delta_{10}}{4}=\frac{3b+4}{2}$ are non-zero square elements in $\mathbb{F}_{q}$. Choose $y,z\in\mathbb{F}_{q}^*$ such that $y^2=\frac{4-3b}{2}$ and $z^2=\frac{3b+4}{2}$. Then we may assume that $x_1=1+y$ and $x_2=-2+z$. It is clear that $y^2+z^2=4$. Since $x_1\in C_{01}$, we have $\eta(x_1+1)=\eta(2+y)=-1$. Since $x_2\in C_{10}$, we have $\eta(x_2)=\eta(-2+z)=-1$, i.e., $\eta(2-z)=1$. By Lemma \ref{20240824lemma3}, this is impossible. Hence $\#A_{01}(b)$ and $\#A_{10}(b)$ cannot both be non-zero.
\end{proof}

Using the quadratic reciprocity law, we can see that
\begin{enumerate}[label=\roman*)]
  \item if $q\equiv 7\ (\mymod 12)$, then $3$ is a non-square element in $\mathbb{F}_{q}$;
  \item if $q\equiv 11\ (\mymod 12)$, then $3$ is a square element in $\mathbb{F}_{q}$.
\end{enumerate}

\begin{lemma}
  $\delta_{F_{2,\frac{1}{3}}}(1,\frac{4}{3})=\begin{cases}2&\mbox{if }q\equiv 7\ (\mymod 24),\\1&\mbox{if }q\equiv 23\ (\mymod 24).\end{cases}$
\end{lemma}

\begin{proof}
  We know that $D_1F_{2,\frac{1}{3}}(0)=\frac{4}{3}$. Since $3\cdot\frac{4}{3}-4=0\not\in C_0$, we have $\# A_{00}(\frac{4}{3})=0$. Since $3\cdot\frac{4}{3}-2=2\not\in C_1$, we have $\# A_{11}(\frac{4}{3})=0$. Note that $A_{10}(\frac{4}{3})=\{x\in C_{10}:\ x^2+4x=0\}\subset\{0,-4\}$. Since $\eta(0)=0$ and $\eta(-4)=-1$, we have
  $$\# A_{10}(\frac{4}{3})=\begin{cases}
      1 & \mbox{if }\eta(3)=-1, \\
      0 & \mbox{if }\eta(3)=1.
    \end{cases}$$
  Finally, note that we have shown that $\# A_{01}(\frac{4}{3})=0$ in the proof of Lemma \ref{20240825lemma12}.
\end{proof}

\begin{lemma}
  $\delta_{F_{2,\frac{1}{3}}}(1,-\frac{2}{3})=\begin{cases}2&\mbox{if }q\equiv 7\ (\mymod 24),\\1&\mbox{if }q\equiv 23\ (\mymod 24).\end{cases}$
\end{lemma}

\begin{proof}
  We know that $D_1F_{2,\frac{1}{3}}(-1)=-\frac{2}{3}$. Since $3\cdot(-\frac{2}{3})+4=2\in C_0$ and $3\cdot(-\frac{2}{3})-4=-6$, we have
  $$\# A_{00}(-\frac{2}{3})=\begin{cases}
      1 & \mbox{if }\eta(3)=-1, \\
      0 & \mbox{if }\eta(3)=1.
    \end{cases}$$
  Since $3\cdot(-\frac{2}{3})+2=0\not\in C_1$, we have $\# A_{11}(-\frac{2}{3})=0$. Note that $A_{10}(-\frac{2}{3})=\{x\in C_{10}:\ x^2+4x+3=0\}\subset\{-1,-3\}$. Since $-1+1=0\not\in C_0$, we have $-1\not\in A_{10}(-\frac{2}{3})$. Since $-3+1=-2\not\in C_0$, we have $-3\not\in A_{10}(-\frac{2}{3})$. Hence $\# A_{10}(-\frac{2}{3})=0$.

  Note that $A_{01}(-\frac{2}{3})=\{x\in C_{01}:\ x^2-2x-2=0\}$. If $\eta(3)=-1$, then $\# A_{00}(-\frac{2}{3})>0$. By Lemma \ref{20240824lemma1}, we have $\# A_{01}(-\frac{2}{3})=0$. Now assume that $\eta(3)=1$. The two roots of $x^2-2x-2$ in $\overline{\mathbb{F}_q}$ are $x_1=1+\sqrt{3}$ and $x_2=1-\sqrt{3}$. Put $a=2$, $u=3$ and $u'=1$. Then $u,u'\in C_0$ and $u+u'=a^2$. Moreover, we have $\eta(a+\sqrt{u'})=\eta(3)=1$ and $\eta(a-\sqrt{u'})=\eta(1)=1$. By Lemma \ref{20240824lemma3}, we have $\eta(a\pm\sqrt{u})=\eta(2\pm\sqrt{3})=1$, which implies that $x_i+1\not\in C_1$ for $i=1,2$. Hence $\# A_{01}(-\frac{2}{3})=0$.
\end{proof}

We can obtain the following corollary from Lemma \ref{20240825lemma9}, Lemma \ref{20240825lemma11} and the above lemmas.

\begin{corollary}
  Assume that $q\equiv 7\ (\mymod\ 8)$. Then for any $b\in\mathbb{F}_q$, $\delta_{F_{2,\frac{1}{3}}}(1,b)\le 3$. Moreover, $\delta_{F_{2,\frac{1}{3}}}(1,b)=3$ if and only if one of the following two cases occurs:
  \begin{enumerate}
    \item $\# A_{10}(b)=2$ and $\# A_{11}(b)=1$ (note that $\# A_{00}(b)$ and $\# A_{01}(b)$ are automatically zero);
    \item $\# A_{00}(b)=\# A_{11}(b)=\# A_{10}(b)=1$ (note that $\# A_{01}(b)$ is automatically zero).
  \end{enumerate}
\end{corollary}

\begin{theorem}\label{20240908them1}
  Assume that $q\equiv 7\ (\mymod\ 8)$ and $q>7$. Then $\delta_{F_{2,\frac{1}{3}}}=3$.
\end{theorem}

\begin{proof}
  Put
  $$\Lambda=\{b\in\mathbb{F}_q:\ \# A_{10}(b)=2\ \mbox{and}\ \# A_{11}(b)=1\}.$$
  We claim that if $q$ is sufficiently large, then $\Lambda\ne\emptyset$.  Note that
  \begin{align*}
    \Lambda & =\{b\in\mathbb{F}_q:\ \# A_{10}(b)=2\ \mbox{and}\ \# A_{11}(b)=1\}          \\
            & =\{b\in\mathbb{F}_q:\ \eta(3b+4)=1,\ \eta(-2\pm\sqrt{\frac{3b+4}{2}})=-1,   \\
            & \qquad\qquad\qquad \eta(-1\pm\sqrt{\frac{3b+4}{2}})=1,\ \eta(3b\pm 2)=-1\}.
  \end{align*}
  Making the substitution $y^2=\frac{3b+4}{2}$, we have
  \begin{align*}
    2\#\Lambda & =\#\left\{y\in\mathbb{F}_q^*:\ \begin{cases} \eta(-2\pm y)=-1,\\ \eta(-1\pm y)=1,\\ \eta(y^2-3)=-1\end{cases}\right\} \\
               & =\#\left\{y\in\mathbb{F}_q^*:\ \begin{cases}\eta(2\pm y)=1,\\ \eta(-1\pm y)=1,\\ \eta(3-y^2)=1\end{cases}\right\}     \\
               & =\frac{1}{32}\sum\limits_{y\in\mathbb{F}_q\setminus A}\displaystyle\prod_{i=1}^5\Big(1+\eta\big(g_i(y)\big)\Big),
  \end{align*}
  where
  \begin{equation*}
    A=\begin{cases}
      \{0,\ \pm 2,\ \pm 1,\ \pm\sqrt{3}\} & \mbox{if }\eta(3)=1,  \\
      \{0,\ \pm 2,\ \pm 1\}               & \mbox{if }\eta(3)=-1,
    \end{cases}
  \end{equation*}
  and
  \begin{equation*}
    \begin{cases}
      g_1(y)=2+y,  & g_2(y)=2-y,  \\
      g_3(y)=-1+y, & g_4(y)=-1-y, \\
      g_5(y)=3-y^2.
    \end{cases}
  \end{equation*}
  Note that $g_4(0)=-1\in C_1$, $g_5(2)=g_5(-2)=-1\in C_1$ and $g_4(1)=g_3(-1)=-2\in C_1$. Moreover, if $\eta(3)=1$, since $g_3(\sqrt{3})g_4(\sqrt{3})=g_3(-\sqrt{3})g_4(-\sqrt{3})=-2\in C_1$, we have $g_3(\sqrt{3})\in C_1$ or $g_4(\sqrt{3})\in C_1$, and $g_3(-\sqrt{3})\in C_1$ or $g_4(-\sqrt{3})\in C_1$. Hence
  $$\sum\limits_{y\in A}\displaystyle\prod_{i=1}^5\Big(1+\eta\big(g_i(y)\big)\Big)=0,$$
  which implies that
  $$64\cdot\#\Lambda=\sum\limits_{y\in\mathbb{F}_q}\displaystyle\prod_{i=1}^5\Big(1+\eta\big(g_i(y)\big)\Big)=\sum\limits_{I}S_I,$$
  where $I$ runs over all subsets of $[5]$ and
  $$S_I=\sum\limits_{y\in\mathbb{F}_q}\eta\big(\displaystyle\prod_{i\in I}g_i(y)\big).$$
  Note that the $g_i$'s have no common roots.
  \begin{enumerate}
    \item If $I=\emptyset$, then $S_I=\sum_{y\in\mathbb{F}_q}\eta(1)=q$.
    \item If $I=\{i\}$ for some $i\in [4]$, then $S_I=\sum_{y\in\mathbb{F}_q}\eta(g_i(y))=0$ since $g_i$ is a linear function. By Lemma \ref{20240825lemma5}, we have $S_{\{5\}}=-\eta(-1)=1$. Hence $\sum_{\# I=1}S_I=1$.
    \item By Lemma \ref{20240825lemma5}, it is easy to see that
          $$\sum\limits_{\substack{I\subset[4],\\\# I=2}}S_I=-2\eta(-1)-\eta(1)-\eta(-1)-\eta(1)-\eta(-1)=2.$$
          By Theorem \ref{20240904them3}, we have $|S_{\{i,5\}}|\le 2\sqrt{q}$ for any $i\in[4]$, which implies that $\sum_{\# I=2}S_I\ge 2-8\sqrt{q}$.
    \item By Theorem \ref{20240904them3}, we have $|S_{I}|\le 2\sqrt{q}$ for any $I\subset [4]$ with $\# I=3$ and $|S_{I\cup\{5\}}|\le 3\sqrt{q}$ for any $I\subset [4]$ with $\# I=2$. It follows that
          $$\sum\limits_{\# I=3}S_I\ge -2\sqrt{q}\cdot\binom{4}{3}-3\sqrt{q}\cdot\binom{4}{2}=-26\sqrt{q}.$$
    \item By Theorem \ref{20240904them3}, we have $|S_{[4]}|\le 3\sqrt{q}$ and $|S_{I\cup\{5\}}|\le 4\sqrt{q}$ for any $I\subset[4]$ with $\# I=3$. It follows that
          $$\sum\limits_{\# I=4}S_I\ge -3\sqrt{q}-4\sqrt{q}\cdot\binom{4}{3}=-19\sqrt{q}.$$
    \item By Theorem \ref{20240904them3}, we have $|S_{[5]}|\le 5\sqrt{q}$.
  \end{enumerate}
  In conclusion, we have
  \begin{align*}
    \sum\limits_{I}S_I & \ge q+1+2-8\sqrt{q}-26\sqrt{q}-19\sqrt{q}-5\sqrt{q} \\
                       & =q-58\sqrt{q}+3.
  \end{align*}
  If $q\ge 58^2$, then $\#\Lambda=\frac{1}{64}\cdot\sum_{I}S_I>0$. For the case where $7<q<58^2$, we directly verify the theorem using a Python program, which can be found in the appendix.
\end{proof}

\begin{remark}
  When $q=7$, we have $\delta_{F_{2,\frac{1}{3}}}=2$, i.e., $F_{2,\frac{1}{3}}$ is an APN function.
\end{remark}

\begin{theorem}\label{20240908them2}
  Assume that $q\equiv 3\ (\mymod\ 8)$, $p\ne 3$ and $q>43$. Then $\delta_{F_{2,\frac{1}{3}}}=4$.
\end{theorem}
\begin{proof}
  We want to show that there exists $b\in\mathbb{F}_q$ such that $\# A_{01}(b)=2$ and $\# A_{11}(b)=\# A_{10}(b)=1$. By (\ref{20240905equation9}), (\ref{20240905equation11}) and (\ref{20240905equation12}), we only need to prove that the following set is non-empty:
  \begin{align*}
    \Lambda & =\left\{b\in\mathbb{F}_q:\ \begin{cases}\eta(3b\pm 2)=-1,                                  \\
                                           \eta(4\pm 3b)=-1,                                  \\  \eta(1\pm\sqrt{\frac{4-3b}{2}})=1,\\\eta(2\pm\sqrt{\frac{4-3b}{2}})=-1,\\
                                           \eta(y-2)=-1, \mbox{where }y\ \mbox{is the square} \\
                                           \ \ \mbox{root of }\frac{3b+4}{2}\ \mbox{such that }\eta(y-1)=1
                                         \end{cases} \right\}.
  \end{align*}
  Making the substitution $y^2=\frac{3b+4}{2}$ and $z^2=\frac{4-3b}{2}$, we have $y^2+z^2=4$ and
  \begin{align*}
    \#\Lambda & =\frac{1}{2}\cdot\#\left\{(y,z)\in{\mathbb{F}_q^*}^2:\ \begin{cases}
                                                                         y^2+z^2=4,       \\                                    \eta(y-2)=-1,    \\
                                                                         \eta(y-1)=1,     \\
                                                                         \eta(1\pm z)=1,  \\
                                                                         \eta(2\pm z)=-1, \\
                                                                         \eta(y^2-1)=1,   \\
                                                                         \eta(y^2-3)=1
                                                                       \end{cases}\right\}       \\
              & =\frac{1}{2}\cdot\#\left\{(y,z)\in{\mathbb{F}_q^*}^2:\ \begin{cases}
                                                                         y^2+z^2=4,       \\                                          \eta(y-2)=-1,    \\
                                                                         \eta(y\pm 1)=1,  \\
                                                                         \eta(1\pm z)=1,  \\
                                                                         \eta(2\pm z)=-1, \\
                                                                         \eta(y^2-3)=1
                                                                       \end{cases}\right\}.
  \end{align*}
  Since $y^2+z^2=4$, by Lemma \ref{20240824lemma3}, we have $\eta(2\pm z)=-1$ if and only if $\eta(2\pm y)=1$. It follows that
  \begin{align*}
    \#\Lambda & =\frac{1}{2}\cdot\#\left\{(y,z)\in{\mathbb{F}_q^*}^2:\ \begin{cases}
                                                                         y^2+z^2=4,      \\                                          \eta(2\pm y)=1,    \\
                                                                         \eta(y\pm 1)=1, \\
                                                                         \eta(1\pm z)=1, \\
                                                                         \eta(y^2-3)=1
                                                                       \end{cases}\right\} \\
              & =\frac{1}{2^8}\cdot\Bigg(\sum\limits_IS_I-\sum\limits_{(y,z)\in A}\displaystyle\prod_{i=1}^7\Big(1+\eta\big(p_i(y,z)\big)\Big)\Bigg)
  \end{align*}
  where $I$ runs over all subsets of $[7]$,
  $$S_I=\sum\limits_{\substack{y,z\in\mathbb{F}_q\\y^2+z^2=4}}\eta\big(\displaystyle\prod_{i\in I}p_i(y,z)\big),$$
  $$A=\begin{cases}
      \big\{(\pm 2,0),(0,\pm 2)\big\}      & \mbox{if }\eta(3)=-1, \\
      \big\{(\pm 2,0),(\pm\sqrt{3},\pm 1), & \mbox{if }\eta(3)=1,  \\
      \ \ (0,\pm 2),(\pm 1,\pm\sqrt{3})\big\}
    \end{cases}$$
  and
  \begin{align*}
     & p_1(y,z)=2+y,   &  & p_2(y,z)=2-y, \\
     & p_3(y,z)=y+1,   &  & p_4(y,z)=y-1, \\
     & p_5(y,z)=y^2-3,                    \\
     & p_6(y,z)=1+z,   &  & p_7(y,z)=1-z,
  \end{align*}
  Since $p_3(-2,0)=p_4(0,-2)=p_7(0,2)=-1\in C_1$, we have
  $$\displaystyle\prod_{i=1}^7\Big(1+\eta\big(p_i(-2,0)\big)\Big)=\displaystyle\prod_{i=1}^7\Big(1+\eta\big(p_i(0,\pm 2)\big)\Big)=0.$$
  Moreover, if $\eta(3)=1$, since $p_3(\sqrt{3},1)\cdot p_4(\sqrt{3},1)=2$ and $\eta(2)=-1$, we have either $p_3(\sqrt{3},1)\in C_1$ or $p_4(\sqrt{3},1)\in C_1$. Hence
  $$\displaystyle\prod_{i=1}^7\Big(1+\eta\big(p_i(\sqrt{3},1)\big)\Big)=0.$$
  Similarly, we can show that
  $$\displaystyle\prod_{i=1}^7\Big(1+\eta\big(p_i(\pm\sqrt{3},\pm1)\big)\Big)=0.$$
  It follows that
  $$\sum\limits_{(y,z)\in A}\displaystyle\prod_{i=1}^7\Big(1+\eta\big(p_i(y,z)\big)\Big)\le 5\cdot 2^6,$$
  which implies that $\#\Lambda\ge\frac{1}{2^8}(\sum_IS_I-320)$. If $I=\emptyset$, then by Lemma \ref{20240903lemma1}, we have
  $$ S_{I}=\#\left\{(y,z)\in\mathbb{F}_q^2:\ y^2+z^2=4\right\}=q+1.$$
  If $I=I^{(1)}=\{1,2\}$, then
  \begin{align*}
    S_I & =\sum\limits_{\substack{y,z\in\mathbb{F}_q                  \\y^2+z^2=4}}\eta\big(4-y^2\big)=\sum\limits_{\substack{y,z\in\mathbb{F}_q \\y^2+z^2=4}}\eta(z^2)\\
        & =\#\left\{(y,z)\in\mathbb{F}_q^2:\ y^2+z^2=4\right\}-2=q-1.
  \end{align*}
  If $I=I^{(2)}=\{5,6,7\}$, then
  \begin{align*}
    S_I & =\sum\limits_{\substack{y,z\in\mathbb{F}_q                    \\y^2+z^2=4}}\eta\big((y^2-3)(1-z^2)\big)=\sum\limits_{\substack{y,z\in\mathbb{F}_q \\y^2+z^2=4}}\eta(y^2-3)^2\\
        & \ge\#\left\{(y,z)\in\mathbb{F}_q^2:\ y^2+z^2=4\right\}-4=q-3.
  \end{align*}
  If $I=I^{(3)}=\{1,2,5,6,7\}$, then
  \begin{align*}
    S_I & =\sum\limits_{\substack{y,z\in\mathbb{F}_q                    \\y^2+z^2=4}}\eta(y^2-3)^2\eta(z)^2\\
        & \ge\#\left\{(y,z)\in\mathbb{F}_q^2:\ y^2+z^2=4\right\}-6=q-5.
  \end{align*}

  Now assume that $\# I\ge 1$ and $I\ne I^{(i)}$ for any $1\le i\le 3$.  We can divide $I$ into two parts: $I=I_1\cup I_2$, where $I_1\subset [5]$ and $I_2\subset\{6,7\}$.
  \begin{enumerate}
    \item If $\# I_2=0$, then $\prod_{i\in I}p_i(y,z)$ is a polynomial of $y$. Denote it by $\gamma(y)$. Then
          \begin{align*}
            S_I & =\sum\limits_{\substack{y,z\in\mathbb{F}_q                                                                           \\y^2+z^2=4}}\eta\big(\gamma(y)\big)\\
                & =\sum\limits_{y\in\mathbb{F}_q}\eta\big(\gamma(y)\big)\Big(1+\eta(4-y^2)\Big)                                        \\
                & =\sum\limits_{y\in\mathbb{F}_q}\eta\big(\gamma(y)\big)+\sum\limits_{y\in\mathbb{F}_q}\eta\big(\gamma(y)(4-y^2)\big).
          \end{align*}
          Since $I\ne I^{(1)}$, neither $\gamma(y)$ nor $\gamma(y)(4-y^2)$ is the square of a polynomial. By Theorem \ref{20240904them3}, we have
          \begin{align*}
            S_I & \ge\big(1-\deg(\gamma)\big)\sqrt{q}+\Big(1-\big(\deg(\gamma)+2\big)\Big)\sqrt{q} \\
                & =-2\deg(\gamma)\sqrt{q}.
          \end{align*}
    \item If $\# I_2=2$, then using the relation $y^2+z^2=4$, $\prod_{i\in I}p_i(y,z)$ can be transformed into a polynomial of $y$. Since $I\ne I^{(i)}$ for $i=2,3$, we can also use Theorem \ref{20240904them3} to give a lower bound for $S_I$.
    \item If $\# I_2=1$, then $\prod_{i\in I}p_i(y,z)$ has the form $\phi(y)(z+a)$, where $\phi\in\mathbb{F}_q[x]$ and $a\in\{\pm 1\}$. Similar to the proof of Theorem \ref{20240905them2} (see (\ref{20240904equation3})), we can show that
          \begin{align*}
            S_I & \ge\big(1-r(\phi)\big)\big(r(\phi)-2\big)\sqrt{q}-5r(\phi)^{\frac{13}{3}}-\deg(\phi)-1,
          \end{align*}
          where $r(\phi):=\max\{4,2+2\deg(\phi)\}$.
  \end{enumerate}
  Using a Python program like the one used in the proof of Theorem \ref{20240905them2}, we can obtain that
  \begin{align*}
    \sum\limits_IS_I & \ge 4q-8-3644\sqrt{q}-5173713 \\
                     & =4q-3644\sqrt{q}-5173721,
  \end{align*}
  which implies that
  $$\#\Lambda\ge\frac{1}{2^8}(4q-3644\sqrt{q}-5174041).$$
  If $q\ge 1681^2$, then $\#\Lambda>0$. For the case where $43<q<1681^2$, we directly verify the theorem using a Python program, which can be found in the appendix.
\end{proof}

\begin{remark}
  When $q\in\{11,19,43\}$, we have $\delta_{F_{2,\frac{1}{3}}}=3$.
\end{remark}

\begin{remark}
  Although $1681^2$ is also a relatively large number, it is still within the range that can be verified. To accelerate the computation, we used multiprocessing techniques. We employed a high-performance computer with $112$ cores and completed the verification within $6$ hours.
\end{remark}

\section{The Differential Spectra and Boomerang Uniformity of $F_{2,\pm 1}$}\label{20240907section5}

This section determines the differential spectra and boomerang uniformity of $F_{2,\pm 1}$. By Lemma \ref{20240825lemma10}, it suffices to consider $F_{2,1}$. We have $D_1F_{2,1}(0)=2$ and $D_1F_{2,1}(-1)=0$.

\noindent{\textbf{Case 1.}}\ If $x\in C_{00}$, then
$$D_1F_{2,1}(x)=2\big((x+1)^2-x^2\big)=4x+2.$$
The unique possible solution of $D_1F_{2,1}(x)=b$ is $x=\frac{b-2}{4}$. Moreover, we have
\begin{equation}
  \# A_{00}(b)=\begin{cases}1&\mbox{if}\ \frac{b-2}{4}\in C_{00},\ \mbox{i.e., }b\pm 2\in C_0,\\0&\mbox{otherwise}.\end{cases}
\end{equation}
It is clear that $\# A_{00}(0)=\# A_{00}(2)=0$.

\noindent{\textbf{Case 2.}}\ If $x\in C_{11}$, then
$$D_1F_{2,1}(x)=(1-1)\big((x+1)^2-x^2\big)=0,$$
which implies that
\begin{equation}
  \# A_{11}(b)=\begin{cases}\# C_{11}=\frac{q-3}{4}&\mbox{if}\ b=0,\\0&\mbox{if }b\ne 0.\end{cases}
\end{equation}
\noindent{\textbf{Case 3.}}\ If $x\in C_{01}$, then $D_1F_{2,1}(x)=-2x^2$. It is clear that $\# A_{01}(0)=0$. Since $\eta(-1)=-1$, we have $\# A_{01}(2)=0$. Assume that $b\ne 0$ and consider the equation $D_1F_{2,1}(x)=b\Leftrightarrow x^2=-\frac{b}{2}$. It is clear that $\# A_{01}(b)\le 1$. Moreover, $\# A_{01}(b)=1$ if and only if $\eta(-\frac{b}{2})=1$ and $\eta(y+1)=-1$, where $y$ is the (only) square root of $-\frac{b}{2}$ such that $\eta(y)=1$.

\noindent{\textbf{Case 4.}}\ If $x\in C_{10}$, then
$$D_1F_{2,1}(x)=2x^2+4x+2=2(x+1)^2.$$
It is clear that $\# A_{10}(0)=0$. The two solutions of $D_{1}F_{2,1}(x)=2$ are $x=0$ and $x=-2$, neither of which is in $C_{10}$. Hence $\# A_{10}(2)=0$. Assume that $b\ne 0$ and consider the equation $D_1F_{2,1}(x)=b\ \Leftrightarrow\ (x+1)^2=\frac{b}{2}$. It is clear that $\# A_{10}(b)\le 1$. Moreover, $\# A_{10}(b)=1$ if and only if $\eta(\frac{b}{2})=1$ and $\eta(y-1)=-1$, where $y$ is the (only) square root of $\frac{b}{2}$ such that $\eta(y)=1$.

In summary, we have $\delta_{F_{2,1}}(1,0)=\# A_{11}(0)+1=\frac{q+1}{4}$, $\delta_{F_{2,1}}(1,2)=1$ and for any $b\in\mathbb{F}_q^*$, $\delta_{F_{2,1}}(1,b)\le 2$. If $q>7$, then $\frac{q+1}{4}>2$. Put
$$\Lambda_1=\left\{b\in\mathbb{F}_q:\ \begin{cases}\eta(b\pm 2)=1, \\ \eta(-\frac{b}{2})=1,\\ \eta(y+1)=-1,\ \mbox{where }y\ \mbox{is the}\\
    \mbox{square root of}\ -\frac{b}{2}\ \mbox{with }\eta(y)=1\end{cases}\right\}$$
and
$$\Lambda_2=\left\{b\in\mathbb{F}_q:\ \begin{cases}\eta(b\pm 2)=1, \\ \eta(\frac{b}{2})=1,\\ \eta(y-1)=-1,\ \mbox{where }y\ \mbox{is the}\\
    \mbox{square root of}\ \frac{b}{2}\ \mbox{with }\eta(y)=1\end{cases}\right\}.$$
Then $\Lambda_1\cup\Lambda_2=\{b\in\mathbb{F}_q:\ \delta_{F_{2,1}}(1,b)=2\}$. We have
\begin{align*}
  \#\Lambda_1 & =\#\left\{y\in\mathbb{F}_q:\ \begin{cases}\eta(y)=1,\\\eta(y+1)=-1,\\\eta(-2y^2\pm 2)=1\end{cases}\right\}                       \\
              & =\#\left\{y\in\mathbb{F}_q:\ \begin{cases}\eta(y)=1,\\\eta(y+1)=-1,\\\eta(y-1)=\eta(2),\\\eta(y^2+1)=-\eta(2)\end{cases}\right\} \\
              & =\frac{1}{16}\Bigg(\sum\limits_{y\in\mathbb{F}_q}\displaystyle\prod_{i=1}^4\Big(1+\eta\big(p_i(y)\big)\Big)                      \\
              & \qquad\qquad-\sum\limits_{y\in A}\displaystyle\prod_{i=1}^4\Big(1+\eta\big(p_i(y)\big)\Big)\Bigg)                                \\
              & =\frac{1}{16}\left(\sum\limits_IS_I-\sum\limits_{y\in A}\displaystyle\prod_{i=1}^4\Big(1+\eta\big(p_i(y)\big)\Big)\right),
\end{align*}
where $A=\{0,\pm1\}$, $I$ runs over all subsets of $[4]$, $S_I=\sum_{y\in\mathbb{F}_q}\eta\big(\prod_{i\in I}p_i(y)\big)$ and
\begin{align*}
   & p_1(y)=y,      &  & p_2(y)=-(y+1),    \\
   & p_3(y)=2(y-1), &  & p_4(y)=-2(y^2+1).
\end{align*}
Since $p_2(0)=-1\in C_1$ and $p_4(1)=p_4(-1)=-4\in C_1$, we have
$$\sum\limits_{y\in A}\displaystyle\prod_{i=1}^4\Big(1+\eta\big(p_i(y)\big)\Big)=0.$$
Now we compute each $S_I$.
\begin{enumerate}
  \item If $I=\emptyset$, then $S_I=\#\mathbb{F}_q=q$.
  \item If $I=\{i\}$ for some $i\in [3]$, then $S_I=0$ since $p_i$ is a linear function. By Lemma \ref{20240825lemma5}, we have $S_{\{4\}}=\eta(2)$. It follows that $\sum_{\#I=1}S_I=\eta(2)$.
  \item By Lemma \ref{20240825lemma5}, we have $S_{\{1,2\}}=1$, $S_{\{1,3\}}=-\eta(2)$ and $S_{\{2,3\}}=\eta(2)$. By Lemma \ref{20240906lemma1}, we have $S_{\{1,4\}}=0$. Hence
        \begin{align*}
          \sum\limits_{\#I=2}S_I & =1+\eta(2)\sum\limits_{y\in\mathbb{F}_q}\eta\big((y+1)(y^2+1)\big) \\
                                 & -\sum\limits_{y\in\mathbb{F}_q}\eta\big((y-1)(y^2+1)\big).
        \end{align*}
  \item By Lemma \ref{20240906lemma1}, we have $S_{\{1,2,3\}}=0$. By Lemma \ref{20240906lemma3}, we have
        $$S_{\{1,2,4\}}=-\eta(2)+\eta(2)\sum\limits_{y\in\mathbb{F}_q}\eta\big((y+1)(y^2+1)\big)$$
        and
        $$S_{\{1,3,4\}}=1+\sum\limits_{y\in\mathbb{F}_q}\eta\big((y-1)(y^2+1)\big).$$
        We have
        $$S_{\{2,3,4\}}=\sum\limits_{y\in\mathbb{F}_q}\eta(y^4-1)=-1$$
        by Lemma \ref{20240908lemma2}.
  \item By Lemma \ref{20240906lemma1}, we have $S_{\{1,2,3,4\}}=0$.
\end{enumerate}
In summary, we have
\begin{align}\label{20240908equation1}
  \#\Lambda_1 & =\frac{1}{16}\Bigg(q+1+2\eta(2)\sum\limits_{y\in\mathbb{F}_q}\eta\big((y+1)(y^2+1)\big)\Bigg).
\end{align}
Similarly, we can prove that
\begin{align}\label{20240908equation2}
  \#\Lambda_2 & =\frac{1}{16}\Bigg(q+1-2\sum\limits_{y\in\mathbb{F}_q}\eta\big((y+1)(y^2+1)\big)\Bigg).
\end{align}

\begin{theorem}
  Assume that $q>7$. The differential spectrum of $F_{2,1}$ is given by
  \begin{align*}
    \begin{cases}\omega_0=\frac{(q-1)\Big(3q-5+\big(\eta(2)-1\big)T\Big)}{8}, \\ \omega_1=\frac{(q-1)\Big(2q-2+\big(1-\eta(2)\big)T\Big)}{4}, \\
      \omega_2=\frac{(q-1)\Big(q+1+\big(\eta(2)-1\big)T\Big)}{8},  \\ \omega_{\frac{q+1}{4}}=q-1,\end{cases}
  \end{align*}
  where
  $$T=\sum\limits_{y\in\mathbb{F}_q}\eta\big((y+1)(y^2+1)\big).$$
  In particular, $F_{2,1}$ is a locally-APN function with differential uniformity $\frac{q+1}{4}$.
\end{theorem}
\begin{proof}
  By Lemma \ref{20240825lemma7}, (\ref{20240908equation1}) and (\ref{20240908equation2}), we have $\omega_{\frac{q+1}{4}}=q-1$
  and
  \begin{align*}
    \omega_2 & =\frac{(q-1)\Big(q+1+\big(\eta(2)-1\big)T\Big)}{8}.
  \end{align*}
  By (\ref{20240622equation1}), we have
  \begin{align*}
    \begin{cases}
      \omega_0+\omega_1+\omega_2+\omega_4=(q-1)q, \\
      \omega_1+2\omega_2+\frac{q+1}{4}\omega_4=(q-1)q.
    \end{cases}
  \end{align*}
  It follows that
  \begin{align*}
    \begin{cases}
      \omega_0=\frac{(q-1)\Big(3q-5+\big(\eta(2)-1\big)T\Big)}{8}, \\
      \omega_1=\frac{(q-1)\Big(2q-2+\big(1-\eta(2)\big)T\Big)}{4}.
    \end{cases}
  \end{align*}
\end{proof}

Finally, we compute the boomerang uniformity of $F_{2,1}$. We need to solve the following system of equations
\begin{align}\label{20240908equation3}
  \begin{cases}
    x^2\big(1+\eta(x)\big)-y^2\big(1+\eta(y)\big)=b, \\
    (x+1)^2\big(1+\eta(x+1)\big)-(y+1)^2\big(1+\eta(y+1)\big)=b
  \end{cases}
\end{align}
for any $b^*\in\mathbb{F}_q$. For any $i,j,k,l\in\{0,1\}$, let $A_{ij,kl}(b)$ be the set of solutions $(x,y)$ of (\ref{20240908equation3}) in $C_{ij}\times C_{kl}$.

\begin{lemma}\label{20240909lemma1}
  For any $b\in\mathbb{F}_q^*$, there is no solution $(x,y)$ to the system of equations (\ref{20240908equation3}) with $x\in\{0,-1\}$ or $y\in\{0,-1\}$.
\end{lemma}
\begin{proof}
  Assume that $x=0$. Then (\ref{20240908equation3}) becomes
  \begin{align*}
    \begin{cases}
      y^2\big(1+\eta(y)\big)=-b, \\
      (y+1)^2\big(1+\eta(y+1)\big)=2-b.
    \end{cases}
  \end{align*}
  Since $b\ne 0$, we have $\eta(y)=1$. If $\eta(y+1)=1$, then we have
  $$(y+1)^2=1-\frac{b}{2}=1+y^2,$$
  which implies that $y=0$ and thus $b=0$. This is a contradiction. If $\eta(y+1)=-1$, then $b=2$ and thus $y^2=-1$. This is also a contradiction. Hence there is no solution $(x,y)$ to (\ref{20240908equation3}) with $x=0$.

  Assume that $x=-1$. Then (\ref{20240908equation3}) becomes
  \begin{align*}
    \begin{cases}
      y^2\big(1+\eta(y)\big)=-b, \\
      (y+1)^2\big(1+\eta(y+1)\big)=-b.
    \end{cases}
  \end{align*}
  Since $b\ne 0$, we have $\eta(y)=1$ and $\eta(y+1)=1$, which implies that $y^2=(y+1)^2=-\frac{b}{2}$. It follows that $y=-\frac{1}{2}$. But then it is impossible that $\eta(y)=\eta(y+1)=1$. Hence there is no solution $(x,y)$ to (\ref{20240908equation3}) with $x=-1$. By symmetry, we can prove the assertion on $y$.
\end{proof}

By Lemma \ref{20240909lemma1}, we have
$$\beta_{F_{2,1}}(1,b)=\sum\limits_{i,j,k,l\in\{0,1\}}\# A_{ij,kl}.$$

\noindent{\textbf{Case 1.}}\ If $(x,y)\in C_{00}\times C_{00}$, then (\ref{20240908equation3}) becomes
\begin{align*}
         & \begin{cases}
             x^2-y^2=\frac{b}{2}, \\
             (x+1)^2-(y+1)^2=\frac{b}{2},
           \end{cases} \\
  \iff\  & \begin{cases}
             (x-y)(x+y)=\frac{b}{2}, \\
             (x-y)(x+y+2)=\frac{b}{2}.
           \end{cases}
\end{align*}
Since $b\ne 0$ we have $x\ne y$, which implies that $x+y=x+y+2$. This is impossible and thus $\# A_{00,00}(b)=0$ for any $b\in\mathbb{F}_q^*$.

\noindent{\textbf{Case 2.}}\ If $(x,y)\in C_{00}\times C_{01}$, then (\ref{20240908equation3}) becomes
\begin{align}\label{20240909equation4}
  \begin{cases}
    x^2-y^2=\frac{b}{2}, \\
    (x+1)^2=\frac{b}{2},
  \end{cases}\iff\ \ \begin{cases}(x+1)^2=\frac{b}{2}, \\
                       y^2=-(2x+1).
                     \end{cases}
\end{align}
It follows that $\# A_{00,01}(b)\le 1$ and $\# A_{00,01}(b)=1$ if and only if
\begin{align*}
  \begin{cases}
    \eta(\frac{b}{2})=1,                                               \\
    \eta(x-1)=1,\ \mbox{where }x\ \mbox{is the (only) square root of}  \\
    \qquad\qquad\qquad\,\frac{b}{2}\ \mbox{such that }\eta(x)=1,       \\
    \eta(1-2x)=1,                                                      \\
    \eta(y+1)=-1,\ \mbox{where }y\ \mbox{is the (only) square root of} \\
    \qquad\qquad\qquad\quad 1-2x\ \mbox{such that }\eta(y)=1.
  \end{cases}
\end{align*}

\noindent{\textbf{Case 3.}}\ If $(x,y)\in C_{00}\times C_{10}$, then (\ref{20240908equation3}) becomes
\begin{align*}
  \begin{cases}
    x^2=\frac{b}{2}, \\
    (x+1)^2-(y+1)^2=\frac{b}{2},
  \end{cases}\iff\ \ \begin{cases}
                       x^2=\frac{b}{2}, \\
                       (y+1)^2=2x+1.
                     \end{cases}
\end{align*}
It follows that $\# A_{00,10}(b)\le 1$ and $\# A_{00,10}(b)=1$ if and only if
\begin{align}\label{20240909equation2}
  \begin{cases}
    \eta(\frac{b}{2})=1,                                               \\
    \eta(x+1)=1,\ \mbox{where }x\ \mbox{is the (only) square root of}  \\
    \qquad\qquad\qquad\,\frac{b}{2}\ \mbox{such that }\eta(x)=1,       \\
    \eta(1+2x)=1,                                                      \\
    \eta(y-1)=-1,\ \mbox{where }y\ \mbox{is the (only) square root of} \\
    \qquad\qquad\qquad\quad 1+2x\ \mbox{such that }\eta(y)=1.
  \end{cases}
\end{align}

\noindent{\textbf{Case 4.}}\ If $(x,y)\in C_{00}\times C_{11}$, then (\ref{20240908equation3}) becomes
\begin{align*}
  \begin{cases}
    x^2=\frac{b}{2}, \\
    (x+1)^2=\frac{b}{2},
  \end{cases}\Rightarrow\ \  x=-\frac{1}{2}.
\end{align*}
Since $\eta(-\frac{1}{2}+1)=\eta(\frac{1}{2})=-\eta(\frac{1}{2})$, we have $-\frac{1}{2}\not\in C_{00}$. Hence $\# A_{00,11}(b)=0$ for any $b\in\mathbb{F}_q^*$.

\noindent{\textbf{Case 5.}}\ If $(x,y)\in C_{01}\times C_{00}$, then (\ref{20240908equation3}) becomes
\begin{align*}
  \begin{cases}
    x^2-y^2=\frac{b}{2}, \\
    (y+1)^2=-\frac{b}{2},
  \end{cases}\iff\ \ \begin{cases}
                       (y+1)^2=-\frac{b}{2}, \\
                       x^2=-(2y+1).
                     \end{cases}
\end{align*}
It follows that $\# A_{00,11}\le 1$ and $\# A_{00,11}=1$ if and only if
\begin{align*}
  \begin{cases}
    \eta(-\frac{b}{2})=1,                                              \\
    \eta(x-1)=1,\ \mbox{where }x\ \mbox{is the (only) square root of}  \\
    \qquad\qquad\qquad\,-\frac{b}{2}\ \mbox{such that }\eta(x)=1,      \\
    \eta(1-2x)=1,                                                      \\
    \eta(y+1)=-1,\ \mbox{where }y\ \mbox{is the (only) square root of} \\
    \qquad\qquad\qquad\quad 1-2x\ \mbox{such that }\eta(y)=1.
  \end{cases}
\end{align*}

\noindent{\textbf{Case 6.}}\ If $(x,y)\in C_{01}\times C_{01}$, then (\ref{20240908equation3}) becomes
\begin{align*}
  \begin{cases}
    x^2-y^2=\frac{b}{2}, \\
    b=0,
  \end{cases}
\end{align*}
which implies that $\# A_{01,01}(b)=0$ for any $b\in\mathbb{F}_q^*$.

\noindent{\textbf{Case 7.}}\ If $(x,y)\in C_{01}\times C_{10}$, then (\ref{20240908equation3}) becomes
\begin{align*}
  \begin{cases}
    x^2=\frac{b}{2}, \\
    (y+1)^2=-\frac{b}{2},
  \end{cases}
\end{align*}
which implies that $\# A_{01,10}(b)=0$ for any $b\in\mathbb{F}_q^*$.

\noindent{\textbf{Case 8.}}\ If $(x,y)\in C_{01}\times C_{11}$, then (\ref{20240908equation3}) becomes
\begin{align*}
  \begin{cases}
    x^2=\frac{b}{2}, \\
    b=0,
  \end{cases}
\end{align*}
which implies that $\# A_{01,11}(b)=0$ for any $b\in\mathbb{F}_q^*$.

\noindent{\textbf{Case 9.}}\ If $(x,y)\in C_{10}\times C_{00}$, then (\ref{20240908equation3}) becomes
\begin{align*}
  \begin{cases}
    y^2=-\frac{b}{2}, \\
    (x+1)^2-(y+1)^2=\frac{b}{2},
  \end{cases}\Leftrightarrow\ \ \begin{cases}
                                  y^2=-\frac{b}{2}, \\
                                  (x+1)^2=2y+1.
                                \end{cases}
\end{align*}
It follows that $\# A_{10,00}\le 1$ and $\# A_{10,00}=1$ if and only if
\begin{align*}
  \begin{cases}
    \eta(-\frac{b}{2})=1,                                              \\
    \eta(x+1)=1,\ \mbox{where }x\ \mbox{is the (only) square root of}  \\
    \qquad\qquad\qquad\,-\frac{b}{2}\ \mbox{such that }\eta(x)=1,      \\
    \eta(1+2x)=1,                                                      \\
    \eta(y-1)=-1,\ \mbox{where }y\ \mbox{is the (only) square root of} \\
    \qquad\qquad\qquad\quad 1+2x\ \mbox{such that }\eta(y)=1.
  \end{cases}
\end{align*}

By following the analysis above, it can be easily proven that $\# A_{10,01}(b)=\# A_{10,10}(b)=\# A_{10,11}(b)=\# A_{11,00}(b)=\# A_{11,01}(b)=\# A_{11,10}(b)=\# A_{11,11}(b)=0$ for any $b\in\mathbb{F}_q^*$. Then we have the following corollary.

\begin{corollary}
  For any $b\in\mathbb{F}_q^*$, we have $\beta_{F_{2,1}}(1,b)=\# A_{00,01}(b)+\# A_{00,10}(b)+\# A_{01,00}(b)+\# A_{10,00}(b)\le 2$. Moreover, $\delta_{F_{2,1}}(1,b)=2$ if and only if $\# A_{00,01}(b)=\# A_{00,10}(b)=1$ or $\# A_{01,00}(b)=\# A_{10,00}(b)=1$.
\end{corollary}

\begin{theorem}
  If $q\ge 9613^2$, then the boomerang uniformity of $F_{2,1}$ is $2$.
\end{theorem}
\begin{proof}
  Put $\Lambda_1=\{b\in\mathbb{F}_q:\ \# A_{00,01}(b)=\# A_{00,10}(b)=1\}$ and $\Lambda_2=\{b\in\mathbb{F}_q:\ \# A_{01,00}(b)=\# A_{10,00}(b)=1\}$. Then it is clear that $\Lambda_1\cup\Lambda_2=\{b\in\mathbb{F}_q:\ \delta_{F_{2,1}}(1,b)=2\}$ and $\Lambda_1=-\Lambda_2$. By (\ref{20240909equation4}) and (\ref{20240909equation2}), $\Lambda_1$ consists of $b\in\mathbb{F}_q$ such that
  \begin{align*}
    \begin{cases}
      \eta(\frac{b}{2})=1,                                                 \\
      \eta(x\pm 1)=1,\ \mbox{where }x\ \mbox{is the (only) square root of} \\
      \qquad\qquad\qquad\,\frac{b}{2}\ \mbox{such that }\eta(x)=1,         \\
      \eta(1\pm 2x)=1,                                                     \\
      \eta(y+1)=-1,\ \mbox{where }y\ \mbox{is the (only) square root of}   \\
      \qquad\qquad\qquad\quad 1-2x\ \mbox{such that }\eta(y)=1,            \\
      \eta(z-1)=-1,\ \mbox{where }z\ \mbox{is the (only) square root of}   \\
      \qquad\qquad\qquad\quad 1+2x\ \mbox{such that }\eta(z)=1.
    \end{cases}
  \end{align*}
  It follows that
  \begin{align*}
    \#\Lambda_1 & =\#\left\{(y,z)\in\mathbb{F}_q^2:\ \begin{cases}
                                                       y^2+z^2=2,               \\
                                                       \eta(y)=1,               \\
                                                       \eta(z)=1,               \\
                                                       \eta(y+1)=-1,            \\
                                                       \eta(z-1)=-1,            \\
                                                       \eta(\frac{z^2-1}{2})=1, \\                                 \eta(\frac{z^2+1}{2})=1, \\
                                                       \eta(\frac{z^2-3}{2})=1
                                                     \end{cases}
    \right\}                                                                                                                              \\
                & =\#\left\{(y,z)\in\mathbb{F}_q^2:\ \begin{cases}
                                                       y^2+z^2=2,              \\
                                                       \eta(y)=1,              \\
                                                       \eta(z)=1,              \\
                                                       \eta(y+1)=-1,           \\
                                                       \eta(1-z)=1,            \\
                                                       \eta(\frac{1+z}{2})=-1, \\                                 \eta(\frac{z^2+1}{2})=1, \\
                                                       \eta(\frac{z^2-3}{2})=1
                                                     \end{cases}
    \right\}.
  \end{align*}
  Using the previous method, we can obtain that
  $$\#\Lambda\ge\frac{1}{2^7}(q-7756\sqrt{q}-17844127).$$
  If $q\ge 9613^2$, then $\#\Lambda>0$.
\end{proof}

\begin{remark}
  Numerical results suggest that the theorem is true when $q\ge 307$.
\end{remark}

\section{Conclusion}\label{20240907section6}

In this paper, we conducted an in-depth study on the differential and boomerang properties of the binomial function $F_{2,u}(x)=x^2\big(1+u\eta(x)\big)$ over $\mathbb{F}_q$, where $q$ is an odd prime power with $q\equiv 3\ (\mymod 4)$ and $u\in\mathbb{F}_q^*$. By adopting a methodology that combines algebraic and geometric tools, we determined the differential uniformity of $F_{2,u}$ for any $u\in\mathbb{F}_q^*$ and specifically proved  that
\begin{align*}
  \delta_{F_{2,u}}=\begin{cases}
                     \frac{q+1}{4} & \mbox{if } u\in\{\pm 1\},                                                       \\
                     5             & \mbox{if }u\in\mathbb{F}_q\setminus\mathcal{U}\ \mbox{and }\eta(1+u)=\eta(u-1), \\
                     4             & \mbox{if }u\in\mathbb{F}_q\setminus\mathcal{U}\ \mbox{and }\eta(1+u)=\eta(1-u), \\
                     4             & \mbox{if }q\equiv 3\ (\mymod 8),\ p\ne 3\ \mbox{and }u\in\{\pm\frac{1}{3}\},    \\
                     3             & \mbox{if }q\equiv 7\ (\mymod 8)\ \mbox{and }u\in\{\pm\frac{1}{3}\},
                   \end{cases}
\end{align*}
where $$\mathcal{U}=\begin{cases}
    \{0,\pm 1\}                & \mbox{if}\ p=3,    \\
    \{0,\pm 1,\pm\frac{1}{3}\} & \mbox{if}\ p\ne 3.
  \end{cases}$$
Note that these equalities hold only when $q$ is sufficiently large. In particular, we disproved the conjecture proposed in \cite{budaghyan2024arithmetization}. We also determined the differential spectra of the locally-APN functions $F_{2,\pm 1}$ by expressing them in terms of several quadratic character sums of cubic polynomials (when $q\equiv 7\ (\mymod 8)$, these character sums are actually eliminated). Finally, we showed that the boomerang uniformity of $F_{2,\pm 1}$ is $2$ for sufficiently large $q$. The proven results show that the function $F_{2,u}$ has favorable differential properties.

The methods used in this paper are both typical and innovative, and we believe they will also help solve other problems.

\bibliographystyle{IEEEtranS}
\bstctlcite{IEEEexample:BSTcontrol}
\bibliography{references}{}
	
\onecolumn
\appendix[Auxiliary Programs used in the proofs]

\lstset{
  columns=fixed,                                       
  numberstyle=\tiny\color{gray},                       
  frame=none,                                          
  backgroundcolor=\color[RGB]{245,245,244},            
  keywordstyle=\color[RGB]{40,40,255},                 
  numberstyle=\footnotesize\color{darkgray},
  commentstyle=\it\color[RGB]{0,96,96},                
  stringstyle=\rmfamily\slshape\color[RGB]{128,0,0},   
  showstringspaces=false,                              
  language=python,                                        
}

Note that all the Python programs in the appendix need to be run in the SageMath environment. SageMath is a free open-source Python-based mathematics software system, whose official website is \url{https://www.sagemath.org}.

\begin{enumerate}
  \item The Program Used in the Proof of Theorem \ref{20240905them2}
        \begin{lstlisting}[language=Python]
import math
# Store the degrees of p_1~p_6
degrees_of_y_polynomials = [1, 1, 1, 1, 2, 2]
# Obtain all subsets of {1,2,3,4,5,6}
I_1s = Subsets({1, 2, 3, 4, 5, 6})
# Store the final results
m_1, m_2 = 0, 0   

def get_deg_omega(deg_phi, deg_rho):
    if deg_rho == 0:
        return max([4, 2 + 2 * deg_phi])
    else:
        return 4+2*deg_phi

def lower_bound_1(deg_phi, deg_rho):
    deg_omega = get_deg_omega(deg_phi, deg_rho)
    return -(deg_omega-1)*(deg_omega-2), -5*pow(deg_omega, 13/3)-deg_phi-1

# If #I_2=0
for I_1 in I_1s:
    if I_1 != set():  # I_1 is not empty
        # In Python, the index starts from 0
        d_gamma = sum([degrees_of_y_polynomials[i-1] for i in I_1])
        m_1 += -2 * d_gamma

# If #I_2=1
for I_1 in I_1s:
    deg_phi = sum([degrees_of_y_polynomials[i-1] for i in I_1])
    differences = lower_bound_1(deg_phi, 0)
    for i in range(4):  # p_7~p_10
        m_1 += differences[0]
        m_2 += differences[1]

# If #I_2=3
for I_1 in I_1s:
    # Note that we should add 2 to the degree
    deg_phi = sum([degrees_of_y_polynomials[i-1] for i in I_1]) + 2
    differences = lower_bound_1(deg_phi, 0)
    for i in range(4):  # 3-sets of {7,8,9,10}
        m_1 += differences[0]
        m_2 += differences[1]

# If #I_2=4
for I_1 in I_1s:
    if I_1 != {5, 6}:
        # Note that we should add 4 to the degree
        d_gamma = sum([degrees_of_y_polynomials[i-1] for i in I_1]) + 4
        # If 5 or 6 is in I_1, no extra roots are added
        if 5 in I_1:
            d_gamma -= 2
        if 6 in I_1:
            d_gamma -= 2
        m_1 += -2 * d_gamma

# If #I_2=2
# I_2={7,8}
for I_1 in I_1s:
    if I_1 != {5}:
        # Note that we should add 2 to the degree
        d_gamma = sum([degrees_of_y_polynomials[i-1] for i in I_1]) + 2
        if 5 in I_1:
            d_gamma -= 2
        m_1 += -2 * d_gamma
# I_2={9,10}
for I_1 in I_1s:
    if I_1 != {6}:
        # Note that we should add 2 to the degree
        d_gamma = sum([degrees_of_y_polynomials[i-1] for i in I_1]) + 2
        if 6 in I_1:
            d_gamma -= 2
        m_1 += -2 * d_gamma
# Otherwise (4 cases)
for I_1 in I_1s:
    deg_phi = sum([degrees_of_y_polynomials[i-1] for i in I_1])
    differences = lower_bound_1(deg_phi, 2)
    for i in range(4):
        m_1 += differences[0]
        m_2 += differences[1]

print(f'm_1={m_1}', '\n', f'm_2={math.floor(m_2)}')
  \end{lstlisting}
  \item The Program Used in the Proofs of Theorem \ref{20240908them1} and Theorem \ref{20240908them2}
        \begin{lstlisting}[language=Python]
import multiprocessing
from collections import Counter
from functools import reduce

# Obtain all prime powers < n
def find_prime_powers(n):
    result = []
    for p in primes(n):
        power = p
        while power < n:
            result.append(power)
            power *= p
    return sorted(result)

def compute_differential_uniformity(prime_power):
    field = GF(prime_power)
    square_elements = set()
    square_table = {}
    for x in field:
        square_table[x] = x^2
        if x != 0:
            square_elements.add(x^2)
    def eta(x):
        if x == 0: return 0
        if x in square_elements: return 1
        return -1
    def F(x):
        return square_table[x] * (1 + field(1)/3 * eta(x))
    
    return max([number for _,number in 
                Counter([F(x+1)-F(x) for x in field]).items()])

# Proof of Theorem 5
print("Proof of Theorem 5")
n_1 = 58**2 
prime_powers_1 = [item for item in find_prime_powers(n_1) if item % 8 == 7]

for prime_power in prime_powers_1:
    differential_uniformity = compute_differential_uniformity(prime_power)
    if differential_uniformity != 3:
        print(f"Exception: q={prime_power}, "
              f"differential uniformity={differential_uniformity}")

# Proof of Theorem 6
print("Proof of Theorem 6")
n_2 = 1681**2
prime_powers_2 = [item for item in find_prime_powers(n_2) 
                  if item % 8 == 3 and item % 3 != 0]

# Split a list into several parts as evenly as possible
def split_list(my_list, count):
    avg_length = len(my_list) // count
    extra = len(my_list) % count
    result = []
    start = 0
    for i in range(count):
        end = start + avg_length + (1 if i < extra else 0)
        result.append(my_list[start:end])
        start = end
    return result

# Obtain the number of CPU cores
num_cores = multiprocessing.cpu_count()
prime_powers_2_split = split_list(prime_powers_2, num_cores)

# Use multithreading to speed up the computation
def worker(prime_powers):
    result=[]
    for prime_power in prime_powers:
        differential_uniformity = compute_differential_uniformity(
                                  prime_power)
        if differential_uniformity != 4:
            result.append((prime_power, differential_uniformity))
    return result
    
with multiprocessing.Pool(processes=num_cores) as pool:
    results = pool.map(worker, prime_powers_2_split)
    results = reduce(lambda x, y: x + y, results)
    For an exception in results:
        print(f"Exception: q={exception[0]}, "
              f"differential uniformity={exception[1]}")
        \end{lstlisting}
\end{enumerate}


\end{document}